\documentclass[11pt,a4paper]{amsart}
\usepackage{amsmath,amssymb,latexsym,amsthm,enumerate}
\usepackage{hyperref} %To be able to click on the links in the paper

\theoremstyle{plain}
\newtheorem{theorem}{Theorem}[section]
\newtheorem{proposition}[theorem]{Proposition}
\newtheorem{lemma}[theorem]{Lemma}
\newtheorem{corollary}[theorem]{Corollary}
\theoremstyle{definition}
\newtheorem{definition}[theorem]{Definition}
\newtheorem{example}[theorem]{Example}
\newtheorem{remark}[theorem]{Remark}
\theoremstyle{remark}

\numberwithin{equation}{section}

\newcounter{numpar}[section]
\renewcommand*{\thenumpar}{\thesection.\arabic{numpar}}
\newcommand*{\newpar}[1]{\refstepcounter{numpar}\medskip\textbf{\thenumpar.#1}}
\newcounter{numitem}[section]
\newcommand*{\newitem}{\refstepcounter{numitem}\medskip\textbf{\thenumitem. }}

\newcommand*{\wt}{\widetilde}
\newcommand*{\ol}{\overline}

\newcommand*{\pn}[1]{\ensuremath{#1}\mbox{-a.s.}}
\newcommand*{\dd}{d}     %For dx, dy, etc.
\newcommand*{\eps}{\varepsilon}

\newcommand*{\cB}{\mathcal B}

\newcommand*{\cF}{\mathcal F}

\newcommand*{\bbN}{\mathbb N}

\newcommand*{\bbR}{\mathbb R}

\newcommand*{\EE}{\mathsf E}
\newcommand*{\PP}{\mathsf P}

\newcommand*{\la}{\langle}
\newcommand*{\ra}{\rangle}
\newcommand*{\loc}{{\mathrm{loc}}}
\DeclareMathOperator{\const}{const}
\DeclareMathOperator{\Law}{Law}
\DeclareMathOperator{\sgn}{sgn}
\DeclareMathOperator{\Var}{Var} %Variation, not variance!

\begin{document}
\title[Loss of the semimartingale property]{On the loss of the semimartingale property\\at the hitting time of a level}

\author{Aleksandar Mijatovi\'{c}}
\address{Department of Mathematics, Imperial College London, London, UK}
\email{a.mijatovic@imperial.ac.uk}

\author{Mikhail Urusov}
\address{Faculty of Mathematics, University of Duisburg-Essen, Essen, Germany}
\email{mikhail.urusov@uni-due.de}

%\thanks will become a 1st page footnote.
\thanks{We are grateful to Francis Hirsch for a helpful discussion.
We thank Nicholas Bingham and the anonymous referee for the comments that helped improve the paper.}



\begin{abstract}
This paper studies %a deterministic characterisation
the loss of the semimartingale property of the process 
$g(Y)$
at the time a one-dimensional diffusion
$Y$ 
hits a level,
where
$g$
is a difference of two convex functions. 
We show that the process 
$g(Y)$
can fail to be a semimartingale 
in two ways only, which leads to a natural
definition of non-semimartingales 
of the \textit{first} and \textit{second kind}.
We give a deterministic if and only if condition
(in terms of 
$g$
and the coefficients of $Y$) for 
$g(Y)$
to fall into 
one of the two classes of processes, which yields 
a characterisation for the loss of the semimartingale 
property. 
A number of applications of the results 
in the theory of stochastic processes and
real analysis are given: e.g. we construct 
an adapted diffusion 
$Y$
on
$[0,\infty)$
and a \emph{predictable} finite stopping time $\zeta$,
such that $Y$ is 
a local semimartingale on the stochastic interval $[0,\zeta)$,
continuous at $\zeta$ and constant after $\zeta$,
but is \emph{not} a semimartingale on $[0,\infty)$.
\end{abstract}

\maketitle

%\tableofcontents

%============================================
\section{Introduction}
\label{sec:sem_mot}
Continuous semimartingales form an important, 
general and well-studied class of stochastic processes. 
This paper deals with the phenomenon of the loss
of the semimartingale property
at the hitting time of a level
as motivated and explained below.

%============================================
\newpar{}
The motivation for this work is best described by the 
following two examples.

\begin{example}
\label{ex:1}
Let $B$ be an $(\cF_t,\PP)$-Brownian motion starting from $x_0>0$ defined on some
filtered probability space $(\Omega,\cF,(\cF_t)_{t\in[0,\infty)},\PP)$.
It is well-known that 
the process 
$\sqrt{|B|}$ 
is not a semimartingale
(see the original reference~\cite{Yor:78}
or the monograph~\cite[Th.~72]{Protter:05}).
A possible short argument is as follows.
Let $X$ be a continuous semimartingale
and $L_{t}^{a}(X)$ its local time at time $t\ge0$
and level $a\in\bbR$.
Recall that if $f$ is a strictly increasing function on $\bbR$,
which moreover is the difference of two convex functions,
then, for any $a\in\bbR$, it holds
$L_{.}^{f(a)}(f(X))=f'_{+}(a)L_{.}^{a}(X)$~a.s.,
where $f'_{+}(a)$ is the right derivative of $f$
at the point $a$
(see \cite[Ch.~VI, Ex.~1.23]{RevuzYor:99}).
If $X:=\sqrt{|B|}$ were a semimartingale,
then, applying the statement above to $f(x)=x^{2}\sgn x$,
we would get that $L_{.}^{0}(|B|)\equiv0$,
which would contradict the well-known fact
that the local time at zero of $|B|$
increases immediately after the time
$\tau^{B}_{0}=\inf\{t\ge0:B_{t}=0\}$.

Intuitively this can be summarized as follows:
the semimartingale property of $\sqrt{|B|}$ fails
\emph{immediately after} $\tau^B_0$
because the increase in local time at zero of $|B|$
and the infinite slope of the function
$x\mapsto\sqrt{x}$
at the origin make the process $\sqrt{|B|}$
accumulate an infinite amount of local time at zero
immediately after~$\tau^B_0$.
\end{example}

It is now natural to ask whether the square root
of a nonnegative continuous semimartingale
that does not accrue local time at zero
may fail to be a semimartingale (for a different reason).
This is also possible as the following example shows.

\begin{example}
\label{ex:2}
Let $x_{0}>0$.
Consider a squared Bessel process
$Y$
of dimension 
$\delta\in(0,1)$
starting from~$x_{0}^{2}$,
i.e. it holds
$\dd Y_{t}=\delta\,\dd t+2\sqrt{Y_{t}}\,\dd W_{t}$,
where $W$ is a Brownian motion.
It is well-known that $Y$ is a nonnegative semimartingale
that a.s. hits~$0$ at a finite time,
$0$~is an instantaneously reflecting boundary point for~$Y$,
and $Y$ does not accrue local time at~$0$.
Let $\rho=(\rho_t)_{t\in[0,\infty)}$
be given by $\rho_t=\sqrt{Y_t}$,
i.e. $\rho$ is a Bessel process of dimension
$\delta\in(0,1)$ starting from $x_0>0$.
It is known that $\rho$ is not a semimartingale.
For completeness we present a formal proof of this fact
in Appendix~\ref{sec:BesNSem}.
Here again the semimartingale property of $\rho$
fails \emph{immediately after}
$\tau^{\rho}_{0}=\inf\{t\ge0:\rho_{t}=0\}$.
\end{example}

As we already observed the loss of the semimartingale
property in both examples above occurs
\emph{immediately after} the hitting time of zero.
Let us first discuss whether this happens
in fact even \emph{at} the hitting time of zero,
i.e. whether the stopped processes
$\sqrt{B^{\tau^{B}_{0}}}$
and $\rho^{\tau^{\rho}_{0}}$
are semimartingales.
We shall see that they are semimartingales
(see Corollaries~\ref{cor:sem_dif1}
and~\ref{cor:sem_dif2}),
i.e. the loss of the semimartingale property
in both examples above does not occur
at the hitting time of zero.

The following natural question arises.

\smallskip\noindent
\textbf{Question~I.}
\emph{Let $B$ be a Brownian motion starting from $x_{0}>0$.
Does there exist a continuous strictly increasing function
$g:[0,\infty)\to\bbR$,
which is smooth on $(0,\infty)$,
such that the process $g(B^{\tau^{B}_{0}})$
is not a semimartingale?}

\smallskip\noindent
In other words we are asking here if the loss of the
semimartingale property can occur \emph{at} $\tau^{B}_{0}$.
The requirement for $g$ to be strictly increasing
stems from the desire to construct
a function ``like $\sqrt{\;\cdot\;}$''.

As we shall see, the answer to Question~I is affirmative,
and we will construct such examples below.

%============================================
\newpar{}
In this paper we consider a one-dimensional diffusion $Y$
with the state space $J=(l,r)$,
$-\infty\le l<r\le\infty$,
possibly exiting its state space at a finite time.
By convention $Y$ is stopped after it reaches $l$ or~$r$.
The setting is formally described in Section~\ref{sec:set}.
Denoting by $\zeta$ the exit time from~$J$
(i.e. the hitting time of either $l$ or~$r$),
we study whether the process 
%given as 
%a function
%of
%$Y$
%(i.e. of the form 
$g(Y)$
loses the semimartingale property at the time~$\zeta$.
A particular case of our discussion, when
$g$
is equal to the identity,
will answer the following question:

\smallskip\noindent
\textbf{Question~II.}
\emph{
Assuming that $Y$ exits $J$ only at finite
endpoints\footnote{Note that
if $Y$ were allowed to exit at an infinite endpoint,
then $Y$ would clearly fail to be a semimartingale.},
can $Y$ fail to be a semimartingale?}

\smallskip
As we shall see, the answer to Question~II is affirmative,
and we will construct examples below.
In particular, our construction
gives rise to a \emph{globally} defined
continuous adapted process
$Y=(Y_t)_{t\in[0,\infty)}$
and a \emph{predictable} stopping time $\zeta$
such that $Y$ is a local semimartingale
on the stochastic interval $[0,\zeta)$,
$Y$ is continuous at $\zeta$ and constant after $\zeta$,
but it is \emph{not} a semimartingale on $[0,\infty)$.
The expression ``$Y$ is a local semimartingale
on $[0,\zeta)$'' means that all stopped processes
$Y^{\tau_n}$ are semimartingales
for some (and then for any) nondecreasing
sequence of stopping times $\{\tau_n\}$
such that $\tau_n\uparrow\zeta$~a.s.
and $\tau_n<\zeta$~a.s.
Note that such a sequence exists
because $\zeta$ is predictable.
This terminology agrees
with~\cite[Def.~4.6]{Sharpe:92}.

At this juncture we refer to
\cite{Maisonneuve:77},
\cite[Sec.~V.1]{Jacod:79},
\cite{Yan:82}, \cite{Zheng:82},
\cite{Sharpe:92},
and~\cite{Sharpe:00},
where several classes of processes
on stochastic intervals
(or even on optional random sets)
are considered.
In particular,
in~\cite{Maisonneuve:77}
(also see~\cite[Ch.~IV, Ex.~1.48]{RevuzYor:99})
the notion of a continuous local martingale
on a stochastic interval $[0,\tau)$
is introduced, where $\tau$ is a stopping time
(not necessarily predictable),
and in~\cite{Yan:82} a way of extending
this notion to c\`adl\`ag processes
is suggested.
An important and delicate point in these 
works is precisely the definition of the
notion of a local martingale on the stochastic
interval 
$[0,\tau)$,
when 
$\tau$
is a non-predictable stopping time. 
From this viewpoint,
our setting, where $\zeta$
is a predictable stopping time,
is simple and unambiguous.
We stress that Question~II
appears not to have been treated 
in these papers.

Finally, we discuss
(omitting certain technical details)
the relations between our treatment
of Question~II and the work in~\cite{Sharpe:00}.
In~\cite{Sharpe:00} a process $X$
on an optional random set $\Lambda$
is considered and the question of interest is 
whether $X$ is a restriction to $\Lambda$
of a globally defined martingale
(this question arises naturally in the setting
of semimartingales on manifolds, when a
semimartingale defined on the entire manifold satisfies the
martingale property on each chart).
The analysis in~\cite{Sharpe:00} is performed under
the standing assumption
that $X$ is the restriction to $\Lambda$
of some special semimartingale.
Hence, our Question~II is precisely
the question of whether this standing assumption holds.
In this paper we give explicit deterministic
if-and-only-if conditions in the diffusion setting
for this assumption to be satisfied 
in the case the optional set is of the form $\Lambda=[0,\zeta)$.
We should, however, note that
the study in~\cite{Sharpe:00}
is particularly interesting when $\Lambda$
is non-predictable. Thus, the present paper
and~\cite{Sharpe:00}, in fact, study
distinct questions tailored
to different settings.

%============================================
\newpar{}
After finishing the paper we discovered
the very deep and surprisingly general
treatment~\cite{CinlarJacodProtterSharpe:80},
where one of the questions discussed is
whether a function of a Markov process
is a semimartingale.
Theorem~4.6 in~\cite{CinlarJacodProtterSharpe:80}
gives a necessary and sufficient condition
for this
in a very general setting.
The Brownian case is discussed in detail
in Section~5 of~\cite{CinlarJacodProtterSharpe:80},
where explicit criteria are presented
for a Brownian motion (Theorems~5.5 and~5.6),
a reflecting Brownian motion (Theorem~5.8),
and a killed Brownian motion (Theorem~5.9).
At the end of Section~5
of~\cite{CinlarJacodProtterSharpe:80},
it is explained how the results for
a Brownian motion can be used
to imply the corresponding results for diffusions
(via a state space transformation
and a random time-change),
but the explicit statements are not presented.

In the present paper, the setting  is far less general setting
than that of Section~4 in~\cite{CinlarJacodProtterSharpe:80}.
As discussed above, we are interested only in
the loss of the semimartingale property
\emph{at} the exit time~$\zeta$.
This allows us to assume from the outset that 
%\textit{a priori}
%That is why we have a priori more restrictions
%on the function $g$ that we apply to
%the possibly exiting $J$ diffusion $Y$
%than in~\cite{CinlarJacodProtterSharpe:80}.
%Namely, we need to assume from the outset that
%that
\begin{equation*}
g\colon J\to\bbR
\text{ is a difference of two convex functions,}
\end{equation*}
which implies that $g(Y)$ is a continuous
semimartingale on the stochastic interval
$[0,\zeta)$, and investigate 
%we are then interested
%in characterising 
the behaviour of $g$ near the endpoints of $J$
that preserves the semimartingale property of
$g(Y)$
globally, i.e. on $[0,\infty)$.
Even though our setting 
is less general
%In spite of this less general setting
than the one in~\cite{CinlarJacodProtterSharpe:80},
the results obtained in this paper are \emph{complementary}
to the results in~\cite{CinlarJacodProtterSharpe:80}.
As explained in more detail below,
we enrich the picture 
presented in~\cite{CinlarJacodProtterSharpe:80}
in several directions.

In Section~\ref{sec:sem_dif} we present
a necessary and sufficient condition
for $g(Y)$ to be a semimartingale
(Theorem~\ref{th:sem_dif1}),
a sufficient one
(Theorem~\ref{th:sem_dif2}),
a necessary one (Theorem~\ref{th:sem_dif3}),
and a discussion of the phenomena that lead
to the loss of the semimartingale property
at~$\zeta$ (Theorem~\ref{th:sem_dif4}).
It may be possible to establish our Theorem~\ref{th:sem_dif1}
%can be obtained 
from general Theorem~4.6 in~\cite{CinlarJacodProtterSharpe:80},
but this way of proving Theorem~\ref{th:sem_dif1}
does not look straightforward.
Furthermore, the authors of~\cite{CinlarJacodProtterSharpe:80}
recommend to obtain results for diffusions
from the corresponding results for Brownian motion,
i.e. from the results of Section~5
in~\cite{CinlarJacodProtterSharpe:80}.
Thus, our Theorem~\ref{th:sem_dif1}
can be deduced from Theorem~5.9
in~\cite{CinlarJacodProtterSharpe:80}
via a state space transformation
and a random time-change.
We, however, prove Theorem~\ref{th:sem_dif1}
directly.
This requires an investigation of the 
convergence of certain additive functionals
of diffusion processes, which is carried out in this paper.
We hope that this classification of convergence obtained here is
of interest in its own right.

The other main results
of Section~\ref{sec:sem_dif},
Theorems~\ref{th:sem_dif2},
\ref{th:sem_dif3}, and~\ref{th:sem_dif4},
do not have their analogues
in~\cite{CinlarJacodProtterSharpe:80}
and thus do not follow from the results
of~\cite{CinlarJacodProtterSharpe:80}.
A question arises why we give a separate
sufficient condition for $g(Y)$
to be a semimartingale
(Theorem~\ref{th:sem_dif2})
and a separate necessary one
(Theorem~\ref{th:sem_dif3})
in the presence of a necessary and sufficient
condition (Theorem~\ref{th:sem_dif1}).
Even though Theorem~\ref{th:sem_dif1}
is a more precise result, it is often
less convenient in specific situations.
For example the sufficient condition
for $g(Y)$ to be a semimartingale
in Theorem~\ref{th:sem_dif2}
is typically easier to verify
than the necessary and sufficient condition
in Theorem~\ref{th:sem_dif1}
(compare~\eqref{eq:sem_dif6} and~\eqref{eq:sem_dif3}).
In specific situations we get some qualitative
information (say, about the structure
of certain examples) from Theorems~\ref{th:sem_dif2}
and~\ref{th:sem_dif3} that is not easy to obtain
from Theorem~\ref{th:sem_dif1}.
For instance,
if one wishes to construct an example
demonstrating that the answer to Question~II
is affirmative, one requires the insight from
Corollary~\ref{cor:sem_dif2}
that the drift has to oscillate around zero
near the finite endpoint, where $Y$ exits.
Corollary~\ref{cor:sem_dif2}
is an immediate consequence of Theorem~\ref{th:sem_dif2}
and does not follow from 
%while the insight described above
%is not seen from 
Theorem~\ref{th:sem_dif1}.

In Section~\ref{sec:sem_ex} we construct examples
answering Questions~I and~II.
For each question we construct two examples:
one for each of the two possible ways
(characterised in Theorem~\ref{th:sem_dif4})
the lose of the semimartingale property can occur.
In Section~\ref{sec:DiscBC} we discuss in more detail
the case where $Y$ is a Brownian motion
stopped upon hitting zero.
We start with two lemmas from real analysis
that arise in the study of the Brownian case
and are also of independent interest.
Then we present a result, Theorem~\ref{th:CharBC1},
where two different equivalent conditions
for $g(Y)$ to be a semimartingale are given.
One of them is a slight variation of the equivalent
condition of Theorem~5.9 in~\cite{CinlarJacodProtterSharpe:80}
(simply put, it is observed that parts (ii) and~(iii)
of Theorem~5.9 in~\cite{CinlarJacodProtterSharpe:80}
imply part~(i) of that theorem). The other one is new.

In Section~\ref{sec:Additive_F}
we consider the additive functional
\begin{equation}
\label{eq:i1}
\int_{J}L_{t}^{y}(Y)\,\nu(\dd y),\quad t\in[0,\zeta],
\end{equation}
where $(L_{t}^{y}(Y);\,t\in[0,\zeta),y\in J)$
is the local time of the diffusion $Y$ and $\nu$ is an arbitrary
positive measure on~$J$.
We describe the stopping time after which
this additive functional is infinite,
and present deterministic
criteria for the convergence and divergence
of~\eqref{eq:i1} at this stopping time.
As a particular case of this investigation,
Lemma~5.10 in~\cite{CinlarJacodProtterSharpe:80}
is generalised to the diffusion setting
and complemented by a criterion
for a.s.-infiniteness of the additive functional.
This characterisation is the reason why 
the idea behind  the proof of the corresponding
result in Section~\ref{sec:Additive_F}
differs from the one 
in~\cite[Lemma~5.10]{CinlarJacodProtterSharpe:80}:
our treatment in Secton~\ref{sec:Additive_F}
uses the Ray-Knight theorem
in the corresponding place.
Finally, in Section~\ref{sec:sem_p}
we prove the theorems from Section~\ref{sec:sem_dif}.

%============================================
\section{Setting and Notations}
\label{sec:set}
\newpar{}
\label{par:set0}
First we introduce some common notations used in the sequel.
Let us consider an open interval $J=(l,r)\subseteq\bbR$.

\noindent
\begin{itemize}
\item
$\ol J$ denotes $[l,r](\subseteq[-\infty,\infty])$.

\item
$\nu_L$ denotes the Lebesgue measure on $J$.

\item
$L^1_\loc(J)$ denotes the set of Borel functions
$J\to[-\infty,\infty]$,
which are locally integrable on~$J$,
i.e. integrable on compact subsets of~$J$
with respect to~$\nu_L$.

\item
For a positive measure $\nu$ on~$J$,
$L^1_\loc(l+,\nu)$ (resp. $L^1_\loc(r-,\nu)$)
denotes the set of Borel functions
$f\colon J\to[-\infty,\infty]$ such that for some $z\in J$,
it holds $\int_{(l,z)}|f(y)|\,\nu(\dd y)<\infty$ 
(resp. $\int_{(z,r)}|f(y)|\,\nu(\dd y)<\infty$).

\item
$L^{1}_{\loc}(l+)$ and $L^{1}_{\loc}(r-)$
denote
$L^{1}_{\loc}(l+,\nu_{L})$ and $L^{1}_{\loc}(r-,\nu_{L})$ respectively.

\item
For a function $x\mapsto f(x)$ on~$J$,
the notations ``$f\in L^{1}_{\loc}(l+,\nu)$''
and ``$f(x)\in L^{1}_{\loc}(l+,\nu)$''
are synonymous.

\item
For a locally finite signed measure $\nu_S$ on~$J$,
$|\nu_S|$ denotes the variation measure of~$\nu_S$.
\end{itemize}

%============================================
\newpar{}
\label{par:set1}
Let the state space be $J=(l,r)$, $-\infty\le l<r\le\infty$, and $Y=(Y_t)_{t\in[0,\infty)}$
be a $J$-valued solution of the one-dimensional SDE
\begin{equation}
\label{eq:set1}
\dd Y_t=\mu(Y_t)\,\dd t+\sigma(Y_t)\,\dd W_t,\quad Y_0=x_0,
\end{equation}
on some filtered probability space $(\Omega,\cF,(\cF_t)_{t\in[0,\infty)},\PP)$,
where $x_0\in J$ and $W$ is an $(\cF_t,\PP)$-Brownian motion.
We allow $Y$ to exit its state space $J$ at a finite time in a continuous way.
The exit time is denoted by $\zeta$.
That is to say, $\PP$-a.s. on $\{\zeta=\infty\}$ the trajectories of $Y$ do not exit~$J$,
while $\PP$-a.s. on $\{\zeta<\infty\}$ we have: either
$\lim_{t\uparrow\zeta}Y_t=r$ or $\lim_{t\uparrow\zeta}Y_t=l$.
Then we need to specify the behaviour of $Y$ after $\zeta$ on $\{\zeta<\infty\}$.
In what follows we assume that on $\{\zeta<\infty\}$ the process $Y$
stays after $\zeta$ at the endpoint of $J$ where it exits,
i.e. $l$ and $r$ are by convention absorbing boundaries.

Throughout the paper it is assumed that the coefficients 
$\mu$ and $\sigma$ in~\eqref{eq:set1} satisfy
the Engelbert-Schmidt conditions
\begin{gather}
\label{eq:set2}
\sigma(x)\ne0\;\;\forall x\in J,\\
\label{eq:set3}
\frac1{\sigma^2},\frac\mu{\sigma^2}\in L^1_\loc(J).
\end{gather}
Under \eqref{eq:set2} and~\eqref{eq:set3} SDE~\eqref{eq:set1}
has a weak solution, unique in law, which possibly exits~$J$
(see~\cite{EngelbertSchmidt:91}
or~\cite[Ch.~5, Th.~5.15]{KaratzasShreve:91}).
Conditions \eqref{eq:set2} and~\eqref{eq:set3}
are reasonable weak assumptions:
any locally bounded Borel function $\mu$
and locally bounded away from zero
Borel function $\sigma$ on $J$ 
satisfy \eqref{eq:set2} and~\eqref{eq:set3}.
In what follows we also need the scale function 
$s$ of $Y$ and its derivative~$\rho$:
\begin{align}
\label{eq:set4}
\rho(x)&=\exp\left\{-\int_c^x \frac{2\mu}{\sigma^2}(y)\,\dd y\right\},\quad x\in J,\\
\label{eq:set5}
s(x)&=\int_c^x \rho(y)\,\dd y,\quad x\in\ol J,
\end{align}
for some $c\in J$. In particular, $s$ is an increasing $C^1$-function $J\to\bbR$
with a strictly positive derivative,
which is absolutely continuous on compact intervals in~$J$,
while $s(r)$ (resp.~$s(l)$) may take value~$\infty$
(resp.~$-\infty$).

%============================================
\section{Characterisation of the Semimartingale Property}
\label{sec:sem_dif}
In this section we study whether $g(Y)$
is a semimartingale for the possibly exiting
diffusion $Y$ described in the previous section
and a certain class of functions $g$ described below.
Let us consider a function $g$ on the state space 
$J$
such that 
\begin{gather}
\label{eq:sem_dif1}
g\colon J\to\bbR\text{ is a difference of two convex functions.}
\end{gather}
In particular, 
the left derivative $g'_-$
and the right derivative
$g'_+$
are well-defined everywhere on~$J$
and are functions of finite variation on
compact subsets of~$J$.
Furthermore the derivative 
$g'$
exists everywhere on 
$J$
except possibly on a countable set.
Therefore the second derivative 
$g''$ exists as a function $\nu_L$-a.e. on~$J$.
It follows from~\eqref{eq:sem_dif1}
that the second derivative of $g$
in the sense of distributions
can be identified with a locally finite signed
measure on~$J$
(see \S~3 in the appendix in~\cite{RevuzYor:99}),
which is typically denoted by $g''(\dd y)$
(see e.g.~\cite[Ch.~VI, Th.~1.5]{RevuzYor:99}).
An equivalent description of this object is as follows:
$g''(\dd y)$ is the locally finite signed measure on $J$
satisfying $g''((a,b])=g'_{+}(b)-g'_{+}(a)$,
$l<a<b<r$.
It follows that the Lebesgue decomposition
of $g''(\dd y)$ with respect to $\nu_{L}$
takes the form
\begin{equation*}
g''(\dd y)=g''(y)\,\dd y+g''_{s}(\dd y),
\end{equation*}
where the locally finite signed measure $g''_{s}(\dd y)$
on $J$ denotes the singular part of $g''(\dd y)$
with respect to~$\nu_{L}$.

In what follows, given a function $g$
satisfying~\eqref{eq:sem_dif1},
we define a locally finite signed measure $\nu_{g}$ on $J$
by the formula
\begin{equation}
\label{eq:nu_g}
\nu_g(\dd y):=
\left(\frac{g'\mu}{\sigma^2}+\frac12g''\right)(y)\dd y
+\frac12g_s''(\dd y).
\end{equation}
Below we use the following terminology:
\begin{align*}
&Y\text{\emph{ exits $J$ at }}r\text{ means }\PP\left(\zeta<\infty,\lim_{t\uparrow\zeta}Y_t=r\right)>0;\\
&Y\text{\emph{ exits $J$ at }}l\text{ is understood in an analogous way.}
\end{align*}
We distinguish between the following four cases:

\smallskip\noindent
(A) $Y$ exits $J$ neither at $l$ nor at~$r$;

\smallskip\noindent
(B) $Y$ exits $J$ at $l$, and there exists a finite limit
\begin{equation*}
g(l):=\lim_{x\downarrow l}g(x);
\end{equation*}
\phantom{(B) }$Y$ does not exit $J$ at~$r$;

\smallskip\noindent
(C) $Y$ exits $J$ at $r$, and there exists a finite limit
\begin{equation*}
g(r):=\lim_{x\uparrow r}g(x);
\end{equation*}
\phantom{(C) }$Y$ does not exit $J$ at~$l$;

\smallskip\noindent
(D) $Y$ exits $J$ at $l$ and at $r$, and there exist finite limits
\begin{equation*}
g(l):=\lim_{x\downarrow l}g(x)\quad\text{and}\quad g(r):=\lim_{x\uparrow r}g(x).
\end{equation*}

\smallskip\noindent
In each of these cases $g(Y)$ is well-defined globally
(i.e. on $[0,\infty)$) and finite,
and hence the question whether
$g(Y)$ is a semimartingale is well-posed. 

\begin{remark}
\label{rem:ChSemPr1}
By the It\^o-Tanaka formula
(see \cite[Ch.~VI, Th.~1.5]{RevuzYor:99}),
condition~\eqref{eq:sem_dif1} implies that
\begin{equation}
\label{eq:ChSemPr1}
(g(Y_{t}))_{t\in[0,\zeta)}
\text{ is a continuous semimartingale on }[0,\zeta).
\end{equation}
In fact, \eqref{eq:sem_dif1} is equivalent
to~\eqref{eq:ChSemPr1}.
In the Brownian case $\mu\equiv0,\sigma\equiv1$
(i.e. $Y$ is a Brownian motion absorbed at $l$ and~$r$)
this follows just as in the proofs
of Theorems~5.5 and~5.6 in~\cite{CinlarJacodProtterSharpe:80}
and is stated right after the proof of Lemma~5.10
in~\cite{CinlarJacodProtterSharpe:80}.
In general it remains to note that
\eqref{eq:sem_dif1} is equivalent to
\begin{equation*}
g\circ s^{-1}\colon s(J)\to\bbR
\text{ is a difference of two convex functions}
\end{equation*}
(under~\eqref{eq:set2} and~\eqref{eq:set3},
both $s$ and $s^{-1}$ are $C^{1}$-functions
with absolutely continuous derivatives
on compact subintervals in~$J$)
and refer to the discussion at the end
of Section~5 in~\cite{CinlarJacodProtterSharpe:80}.
Therefore, %\eqref{eq:sem_dif1} is equivalent to
since condition~\eqref{eq:ChSemPr1}
is necessary for $g(Y)$
to be a semimartingale globally (i.e. on $[0,\infty)$),
assuming~\eqref{eq:sem_dif1}
and studying whether $g(Y)$ is a semimartingale
amounts to studying whether the loss
of the semimartingale property occurs
\emph{at} the time~$\zeta$.
\end{remark}

\textbf{Case (A).}
There is nothing to study in this case: $g(Y)$ is always a semimartingale.

\smallskip
\textbf{Case (B).}
First let us note that by Propositions~\ref{prop:set3}--\ref{prop:set5}, case~(B) amounts to the following:

\smallskip\noindent
(B.i) there is a finite limit $g(l):=\lim_{x\downarrow l}g(x)$;

\noindent
(B.ii) $s(l)>-\infty$ and
$\frac{s-s(l)}{\rho\sigma^2}\in L^1_\loc(l+)$;

\noindent
(B.iii) either $s(r)=\infty$ or it holds:
\begin{equation*}
s(r)<\infty\quad\text{and}\quad
\frac{s(r)-s}{\rho\sigma^2}\notin L^1_\loc(r-).
\end{equation*}

\begin{theorem}
\label{th:sem_dif1}
Assume \eqref{eq:sem_dif1} and case~(B).
Then $g(Y)$ is a semimartingale if and only if
\begin{equation}
\label{eq:sem_dif3}
\frac{s-s(l)}{\rho}\in L^1_\loc(l+,|\nu_g|),
\end{equation}
where the variation measure 
$|\nu_g|$
of the locally finite signed measure 
$\nu_g$, defined in~\eqref{eq:nu_g},
equals
\begin{equation*}
|\nu_g|(\dd y)=
\left|\frac{g'\mu}{\sigma^2}+\frac12g''\right|(y)\dd y+
\frac12|g_s''|(\dd y).
\end{equation*}
\end{theorem}

\begin{remark}
\label{rem:sem_dif1}
The proof of Theorem~\ref{th:sem_dif1}
will reveal 
that under~\eqref{eq:sem_dif3},
$g(Y)$ has the semimartingale decomposition
\begin{equation}
\label{eq:sem_dif3.5}
g(Y_t)=g(x_0)+A_t+M_t,\quad t\in[0,\infty),
\end{equation}
where
\begin{align}
\label{eq:sem_dif4}
A_t&=\int_JL_{t\wedge\zeta}^y(Y)\,\nu_g(\dd y),\quad t\in[0,\infty),\\
\label{eq:sem_dif5}
M_t&=\int_0^{t\wedge\zeta}(g'\sigma)(Y_u)\,\dd W_u,\quad t\in[0,\infty),
\end{align}
and the integrals in \eqref{eq:sem_dif4} and~\eqref{eq:sem_dif5} are well-defined.
The random field 
$\{L_t^y(Y):y\in J, t\in[0,\zeta)\}$
in~\eqref{eq:sem_dif4} denotes the local time
of the semimartingale 
$Y$
defined on the stochastic interval 
$[0,\zeta)$
(see Section~\ref{sec:Additive_F}
for further details and references on local time of
$Y$).
Note also that the local martingale
$M$
in~\eqref{eq:sem_dif5} does not depend
on the choice of 
$g'$
on any countable set.
In particular, on the set where the
left and the right derivatives of
$g$
do not coincide 
we can define 
$g'$
arbitrarily. 
%where $g'$ does not exist.
\end{remark}

In the case the measure 
$g''(\dd y)$
is absolutely continuous with respect to 
$\nu_L$,
Theorem~\ref{th:sem_dif1}
implies the following
characterisation.

\begin{corollary}
\label{cor:sem_dif_twice_diff}
Assume 
$g\in C^{1}(J,\bbR)$
and that 
$g'$
is absolutely continuous on compact intervals in~$J$.
Then, in case~(B), it holds
that 
$g(Y)$ is a semimartingale if and only if
\begin{equation*}
\frac{s-s(l)}{\rho}\left|\frac{g'\mu}{\sigma^2}+\frac12g''\right|\in L^1_\loc(l+).
\end{equation*}
\end{corollary}

\begin{remark}
\label{rem:sem_dif1_cor}
Under the assumptions of Corollary~\ref{cor:sem_dif_twice_diff},
the signed measure
$g''_s(\dd y)$
is a zero measure
and  the finite variation process in the semimartingale decomposition~\eqref{eq:sem_dif3.5}
takes the form
\begin{equation*}
A_t=\int_0^{t\wedge\zeta}
\left(g'\mu+\frac12g''\sigma^2\right)(Y_u)
\,\dd u,\quad t\in[0,\infty).
\end{equation*}
\end{remark}

We now investigate when the process $Y$ itself
is a semimartingale.
To get a deterministic necessary and sufficient
condition it is now enough to apply
Theorem~\ref{th:sem_dif1}
or Corollary~\ref{cor:sem_dif_twice_diff}
with $g(x)=x$, $x\in J$.

\begin{corollary}
\label{cor:semdif1cor}
Assume that $l>-\infty$,
$Y$ exits $J$ at~$l$,
$Y$ does not exit $J$ at~$r$.
Then $Y$ is a semimartingale if and only if
\begin{equation*}
\frac{s-s(l)}\rho\frac{|\mu|}{\sigma^{2}}\in L^{1}_{\loc}(l+).
\end{equation*}
\end{corollary}

In specific examples it may be hard to
check~\eqref{eq:sem_dif3}.
The following result, Theorem~\ref{th:sem_dif2},
gives an easy-to-check sufficient condition
for $g(Y)$ to be a semimartingale.
In Theorem~\ref{th:sem_dif3} below we  present a necessary condition for the semimartingale property of~$g(Y)$.

\begin{theorem}
\label{th:sem_dif2}
In addition to the assumptions of Theorem~\ref{th:sem_dif1} suppose that, for some $a\in J$,
\begin{equation}
\label{eq:sem_dif6}
\text{either}\quad 
\nu_g\vert_{(l,a)} \text{ is a positive}
\quad\text{or}\quad 
\nu_g\vert_{(l,a)} \text{ is a negative}\quad\text{measure}.
\end{equation}
Then $g(Y)$ is a semimartingale.
\end{theorem}

\begin{remark}
\label{rem:sem_dif2}
\begin{enumerate}[(i)]
\item
In view of Theorem~\ref{th:sem_dif1},
there is an equivalent reformulation of Theorem~\ref{th:sem_dif2},
which appears to be purely analytic:
\emph{under the assumptions of Theorem~\ref{th:sem_dif1}, \eqref{eq:sem_dif6} implies~\eqref{eq:sem_dif3}.}
Let us observe that our proof is probabilistic and,
furthermore, the task of finding an analytic proof does not seem to be straightforward.

\item
Observe that \eqref{eq:sem_dif3}
does not imply~\eqref{eq:sem_dif6}.
For instance,
consider $J=(0,\infty)$, $\mu\equiv0$, $\sigma\equiv1$,
$g(x)=\int_1^x(2+\sin\frac1{\sqrt{y}})\,\dd y$,
$x\in[0,\infty)$.
\end{enumerate}
\end{remark}

\begin{corollary}
\label{cor:sem_dif1}
In addition to the assumptions of Theorem~\ref{th:sem_dif1} suppose that, for some $a\in J$,
\begin{equation*}
\mu=0\quad\nu_L\text{-a.e. on }(l,a)
\end{equation*}
and
\begin{equation*}
g\text{ is convex or concave on }(l,a).
\end{equation*}
Then $g(Y)$ is a semimartingale.
\end{corollary}

In particular, it immediately follows from
Corollary~\ref{cor:sem_dif1}
that $\sqrt{B^{\tau^{B}_{0}}}$
is a semimartingale
(see the discussion after Examples~\ref{ex:1}
and~\ref{ex:2}).
This can also be seen 
%but the latter is easy to see also 
directly since, by Jensen's inequality,
the process
$\sqrt{B^{\tau^{B}_{0}}}$ 
is a supermartingale.

\begin{remark}
\label{rem:sem_dif3}
\begin{enumerate}[(i)]
\item\label{it:i}
Let $X$ be a continuous semimartingale satisfying 
$\PP(X_t\ge l\;\forall t\ge0)=1$ for some $l>-\infty$,
and $h\colon[l,\infty)\to\bbR$
a convex or concave function continuous at~$l$
with a finite derivative~$h'(l+)$.
Then $h(X)$ is a semimartingale by the It\^o-Tanaka formula
because such a function~$h$ can be extended
to a convex or concave function on~$\bbR$.
However, if $|h'(l+)|=\infty$, the It\^o-Tanaka formula
cannot be used to conclude that $h(X)$ is a semimartingale
(recall Examples~\ref{ex:1} and~\ref{ex:2},
where the semimartingale property is lost
for $h(\cdot)=\sqrt{\:\cdot\;}$, $l=0$).

\item
The statement in~\eqref{it:i} demonstrates that the gist of Corollary~\ref{cor:sem_dif1}
lies in the cases ${|g'(l+)|=\infty}$
or $l=-\infty$.
\end{enumerate}
\end{remark}

We now apply Theorem~\ref{th:sem_dif2}
to get a sufficient condition for $Y$ itself
to be a semimartingale.

\begin{corollary}
\label{cor:sem_dif2}
Assume that $l>-\infty$, $Y$ exits $J$ at~$l$, $Y$ does not exit $J$ at~$r$.
Further suppose that, for some $a\in J$,
\begin{equation*}
\text{either}\quad\mu\ge0\quad\nu_L\text{-a.e. on }(l,a)
\quad\text{or}\quad\mu\le0\quad\nu_L\text{-a.e. on }(l,a).
\end{equation*}
Then $Y$ is a semimartingale.
\end{corollary}

In particular, it follows from Corollary~\ref{cor:sem_dif2}
that $\rho^{\tau^{\rho}_{0}}$ is a semimartingale
(see the discussion after Examples~\ref{ex:1} and~\ref{ex:2}).
Indeed, by It\^o's formula,
on the stochastic interval $[0,\tau^{\rho}_{0})$
it holds
$\dd\rho_{t}=\frac{\delta-1}{2\rho_{t}}dt+dW_{t}$,
hence Corollary~\ref{cor:sem_dif2}
applies with $J=(0,\infty)$,
$\sigma\equiv1$,
$\mu(y)=\frac{\delta-1}{2y}\le0$, $y\in J$.

It is interesting to note that even though
Corollary~\ref{cor:semdif1cor}
gives a more precise result than
Corollary~\ref{cor:sem_dif2},
the latter is sometimes more convenient.
For instance, we can conclude from Corollary~\ref{cor:sem_dif2}
(but not from Corollary~\ref{cor:semdif1cor})
that for $Y$ to fail the semimartingale property,
the drift $\mu$ has to oscillate around zero
near the boundary point~$l$.
Such examples will be constructed below.

We now present a necessary condition for $g(Y)$ to be a semimartingale.

\begin{theorem}
\label{th:sem_dif3}
Under the assumptions of Theorem~\ref{th:sem_dif1} let $g(Y)$ be a semimartingale. Then
\begin{equation}
\label{eq:sem_dif7}
\frac{s-s(l)}{\rho}(g')^2\in L^1_\loc(l+).
\end{equation}
\end{theorem}

Put differently, if \eqref{eq:sem_dif7} is violated, then $g(Y)$ is not a semimartingale.
Let us note that in specific situations it may be easier to see that \eqref{eq:sem_dif7}
is violated than that \eqref{eq:sem_dif3} is violated.

\begin{remark}
\label{rem:sem_dif4}
In the language of analysis, Theorem~\ref{th:sem_dif3} can be recast as follows:
\emph{under the assumptions of Theorem~\ref{th:sem_dif1}, \eqref{eq:sem_dif3} implies~\eqref{eq:sem_dif7}.}
Again we observe that our proof is probabilistic
and that an analytic proof appears not to be straightforward.
Note also that~\eqref{eq:sem_dif7} does not in general imply~\eqref{eq:sem_dif3}
(see Example~\ref{ex:sem_ex1} below).
\end{remark}

Finally, we characterise the phenomena that lead to the loss of the semimartingale property of~$g(Y)$.
As in~\cite{JacodShiryaev:03} we will denote
by $\Var A=(\Var A_t)_{t\in[0,\infty)}$
the variation process of a process $A=(A_t)_{t\in[0,\infty)}$.
Let the assumptions of Theorem~\ref{th:sem_dif1} hold
(in particular, $\PP(\zeta<\infty)>0$)
and let $g(Y)$ be a non-semimartingale.
Decomposition~\eqref{eq:sem_dif3.5}
with $A$
and $M$
given by~\eqref{eq:sem_dif4} and~\eqref{eq:sem_dif5}
still holds, but only on the stochastic interval $[0,\zeta)$
(also $A=(A_{t})_{t\in[0,\zeta)}$
and $M=(M_{t})_{t\in[0,\zeta)}$
are in general well-defined only on $[0,\zeta)$,
$A$ has a locally finite variation on $[0,\zeta)$,
$M$ is a local martingale on $[0,\zeta)$).
We use this decomposition on the stochastic interval
$[0,\zeta)$
in the following definition.

\begin{definition}
\label{def:sem_dif1}
Let the assumptions of Theorem~\ref{th:sem_dif1} hold
and let $g(Y)$ be a non-semimartingale.

(i) We say that $g(Y)$ is a non-semimartingale
\emph{of the first kind}
if $\PP$-a.s. on $\{\zeta<\infty\}$ there are finite limits
\begin{equation*}
M_{\zeta}=\lim_{t\uparrow\zeta}M_{t}
\quad\text{and}\quad
A_{\zeta}=\lim_{t\uparrow\zeta}A_{t}.
\end{equation*}

(ii) We say that $g(Y)$ is a non-semimartingale
\emph{of the second kind}
if $\PP$-a.s. on $\{\zeta<\infty\}$ it holds
\begin{equation*}
\limsup_{t\uparrow\zeta}M_t=-\liminf_{t\uparrow\zeta}M_t=\infty
\quad\text{and}\quad
\limsup_{t\uparrow\zeta}A_t=-\liminf_{t\uparrow\zeta}A_t=\infty.
\end{equation*}
\end{definition}

We will now see that $g(Y)$ can lose
the semimartingale property in these two ways only.
Moreover, we have the following characterisation result.

\begin{theorem}
\label{th:sem_dif4}
Let the assumptions of Theorem~\ref{th:sem_dif1} hold.

(i) $g(Y)$ is a non-semimartingale of the first kind
if and only if \eqref{eq:sem_dif7} holds
and \eqref{eq:sem_dif3} is violated.
In this case the process
$(M_{t\wedge\zeta})_{t\in[0,\infty)}$
is a continuous local martingale on $[0,\infty)$
(not only on $[0,\zeta)$), but
$\Var A_{\zeta}=\infty$ $\PP$-a.s. on~$\{\zeta<\infty\}$.

(ii) $g(Y)$ is a non-semimartingale of the second kind
if and only if \eqref{eq:sem_dif7} is violated.
\end{theorem}

\smallskip
\textbf{Cases (C) and (D)}
are treated similarly to case~(B).
For instance, the counterpart of Theorem~\ref{th:sem_dif1} in case~(D) is as follows:
under \eqref{eq:sem_dif1},
$g(Y)$ is a semimartingale if and only if
\begin{equation*}
\frac{s-s(l)}{\rho}\in L^1_\loc(l+,|\nu_g|)
\quad\text{and}\quad
\frac{s(r)-s}{\rho}\in L^1_\loc(r-,|\nu_g|).
\end{equation*}
We omit further details.

%============================================
\section{Examples}
\label{sec:sem_ex}
\newpar{ Answer to Question~I.}
\label{par:sem_ex1}
Let $B$ be an $(\cF_t,\PP)$-Brownian motion starting from $x_0>0$
defined on some filtered probability space $(\Omega,\cF,(\cF_t)_{t\in[0,\infty)},\PP)$.
Question~I in the introduction asks
whether it is possible
to find a function $g\colon[0,\infty)\to\bbR$ satisfying
\begin{equation}
\label{eq:sem_ex1}
g\in C([0,\infty),\bbR)\cap C^\infty((0,\infty),\bbR)
\end{equation}
and
\begin{equation}
\label{eq:sem_ex2}
g\text{ is strictly increasing}
\end{equation}
such that $(g(B_{t\wedge\tau^B_0}))_{t\in[0,\infty)}$ is not a semimartingale,
where $\tau^{B}_{0}=\inf\{t\ge0:B_{t}=0\}$.
Following the discussion at the end
of Section~\ref{sec:sem_dif}
(see in particular Definition~\ref{def:sem_dif1} and Theorem~\ref{th:sem_dif4}),
wo further natural subquestions arise:

\smallskip\noindent
(a) Can $g(B^{\tau_0^B})$ be a non-semimartingale of the first kind?

\smallskip\noindent
(b) Can $g(B^{\tau_0^B})$ be a non-semimartingale of the second kind?

\smallskip\noindent
The present setting here is a special case of the setting in Section~\ref{sec:sem_dif}
with $J=(0,\infty)$, $\mu\equiv0$, $\sigma\equiv1$, and we are in case~(B)
(note that condition~\eqref{eq:sem_dif1} and the existence of a finite limit
$g(0):=\lim_{x\downarrow0}g(x)$ hold due to~\eqref{eq:sem_ex1}).
Conditions \eqref{eq:sem_dif3} and~\eqref{eq:sem_dif7} take  the form
\begin{equation}
\label{eq:sem_ex3}
x|g''(x)|\in L^1_\loc(0+)
\end{equation}
and
\begin{equation}
\label{eq:sem_ex4}
x(g'(x))^2\in L^1_\loc(0+)
\end{equation}
respectively.
Thus, question~(a) above amounts to constructing a function $g\colon[0,\infty)\to\bbR$
satisfying \eqref{eq:sem_ex1}, \eqref{eq:sem_ex2} and~\eqref{eq:sem_ex4}, such that \eqref{eq:sem_ex3} is violated;
question~(b) amounts to constructing a function $g$ satisfying \eqref{eq:sem_ex1} and~\eqref{eq:sem_ex2},
such that \eqref{eq:sem_ex4} is violated. The answers to both questions (a) and~(b) are affirmative.
We now construct both examples.

\begin{example}[$g(B^{\tau^B_0})$ is a non-semimartingale of the first kind]
\label{ex:sem_ex1}
\mbox{}\\
Let us consider the function $h\colon(0,\infty)\to\bbR$ given by
\begin{equation*}
h(x)=\frac1{\sqrt{x}}\left(2+\sin\frac1x\right),\quad x\in(0,\infty).
\end{equation*}
It is easy to see that $h$ satisfies
\begin{align}
\label{eq:sem_ex5}
h&\in C^\infty((0,\infty),\bbR),\\
\label{eq:sem_ex6}
h(x)&>0\;\;\forall x\in(0,\infty),\\
\label{eq:sem_ex7}
h&\in L^1_\loc(0+),\\
%\notag
xh^2(x)&\in L^1_\loc(0+),\\
%\notag
x|h'(x)|&\notin L^1_\loc(0+).
\end{align}
Setting
\begin{equation*}
g(x)=\int_1^x h(y)\,\dd y,\quad x\in[0,\infty)
\end{equation*}
(note that $g(0)$ is finite due to~\eqref{eq:sem_ex7}),
we get a function $g$ satisfying \eqref{eq:sem_ex1}, \eqref{eq:sem_ex2}, and~\eqref{eq:sem_ex4}
such that \eqref{eq:sem_ex3} is violated, which is what was required.
\end{example}

\begin{example}[$g(B^{\tau^B_0})$ is a non-semimartingale of the second kind]
\label{ex:sem_ex2}
\mbox{}\\
Let us set
\begin{align*}
a_n&=\frac1n-\frac1{n^4},\quad n=2,3,\ldots,\\
b_n&=\frac1n+\frac1{n^4},\quad n=2,3,\ldots,\\
E&=\bigcup_{n=2}^\infty(a_n,b_n)
\end{align*}
and define the strictly positive function
\begin{equation*}
\ol h(x)=\begin{cases}
\frac1{x^2}&\text{if }x\in E,\\
\frac1{\sqrt{x}}&\text{if }x\in(0,\infty)\setminus E.
\end{cases}
\end{equation*}
Since $\int_{a_n}^{b_n}\frac{\dd x}{x^2}=\frac{b_n-a_n}{a_nb_n}\sim\frac{\const}{n^2}$ as $n\to\infty$,
we get $\ol h\in L^1_\loc(0+)$.
It follows from
$\int_{a_n}^{b_n}\frac{\dd x}{x^3}\ge\frac1{b_n}\int_{a_n}^{b_n}\frac{\dd x}{x^2}\sim\frac{\const}{b_nn^2}\sim\frac{\const}n$
as $n\to\infty$ that $x\ol h^2(x)\notin L^1_\loc(0+)$.
It is clear that such a function $\ol h$ can be smoothened in the neighbourhoods of the points $a_n$ and $b_n$, $n=2,3,\ldots$,
so that we get a function $h\colon(0,\infty)\to\bbR$ satisfying \eqref{eq:sem_ex5}--\eqref{eq:sem_ex7} and
\begin{equation*}
xh^2(x)\notin L^1_\loc(0+).
\end{equation*}
Setting
\begin{equation*}
g(x)=\int_1^xh(y)\,\dd y,\quad x\in[0,\infty),
\end{equation*}
we get a function $g$ satisfying \eqref{eq:sem_ex1} and~\eqref{eq:sem_ex2}
such that \eqref{eq:sem_ex4} is violated.
\end{example}

%============================================
\newpar{ Answer to Question~II.}
\label{par:sem_ex2}
Let us consider the setting and notation of Section~\ref{sec:set}.
Question~II in the introduction asks
whether $Y$ can fail to be a semimartingale
whenever $Y$ exits $J$ only at finite endpoints.
Let us consider case~(B) of Section~\ref{sec:sem_dif}
with $l>-\infty$ and $g(x)=x$, $x\in J$.
Now two further natural subquestions arise:

\smallskip\noindent
(c) Can $Y$ be a non-semimartingale of the first kind?

\smallskip\noindent
(d) Can $Y$ be a non-semimartingale of the second kind?

\smallskip\noindent
The answers to both questions are affirmative.
The examples are obtained from Examples \ref{ex:sem_ex1} and~\ref{ex:sem_ex2}
by setting $J:=(g(0),g(\infty))$ and $Y:=g(B^{\tau^B_0})$
(that is,
$\mu=\frac12g''\circ g^{-1}$, $\sigma=g'\circ g^{-1}$).

%============================================
\section{Further Discussions in the Brownian Case}
\label{sec:DiscBC}
In this section we discuss in more detail
the particular case,
where $Y$ is a Brownian motion
stopped upon hitting zero,
i.e. the case $J=(0,\infty)$,
$\mu\equiv0$, $\sigma\equiv1$.

%============================================
\newpar{ Two Lemmas from Real Analysis.}
\label{par:LemRA}
We will need the following result from real analysis,
which is also of independent interest. 

\begin{lemma}
\label{lem:LemRA1}
For some $a>0$, let
\begin{gather}
\label{eq:LemRA1}
g\colon(0,a)\to\bbR
\text{ be a difference of two convex functions,}\\
\label{eq:LemRA2}
\int_{(0,u]}x\,|g''|(\dd x)<\infty
\end{gather}
for some $u\in(0,a)$. Then
\begin{gather}
\label{eq:LemRA3}
\text{there exists a finite limit }
g(0):=\lim_{x\downarrow0}g(x),\\
\label{eq:LemRA4}
\int_{(0,u]}x(g'(x))^{2}\,\dd x<\infty.
\end{gather}
\end{lemma}

Let us recall that $g''(\dd x)$
is the locally finite signed measure on $(0,a)$
satisfying
$g''((x,y])=g'_{+}(y)-g'_{+}(x)$,
$0<x<y<a$,
and $|g''|(\dd x)$ is the variation measure of~$g''(\dd x)$.
Let us further note that statement~\eqref{eq:LemRA4}
does not depend on the definition of the integrand
on the (at most countable) set
where $g'$ does not exist.
For more details, see the discussion
in the beginning of Section~\ref{sec:sem_dif}.

Let us observe that Lemma~\ref{lem:LemRA1}
is a refinement of the analytical statement
implied by Theorems~\ref{th:sem_dif1}
and~\ref{th:sem_dif3}
in the Brownian case.
Indeed, Remark~\ref{rem:sem_dif4} states that
\eqref{eq:LemRA1}--\eqref{eq:LemRA3}
imply~\eqref{eq:LemRA4}.
Note that \eqref{eq:LemRA3} is assumed
in Theorems~\ref{th:sem_dif1}
and~\ref{th:sem_dif3}
as a part of the description
of case~(B) in Section~\ref{sec:sem_dif}.

\begin{proof}
First we prove by contradiction
that \eqref{eq:LemRA1} and~\eqref{eq:LemRA2}
imply~\eqref{eq:LemRA3}.
If this were not true,
there would exist a convex function $h$
on $(0,a)$ such that
\begin{equation}
\label{eq:LemRA5}
\int_{(0,u]}x\,h''(\dd x)<\infty
\end{equation}
and
\begin{equation}
\label{eq:LemRA6}
\lim_{x\downarrow0}h(x)=\infty
\end{equation}
(note that for a convex function
such a limit always exists
but may be infinite).
For~$\eps\in(0,u)$, integrating by parts, we get
\begin{equation*}
\int_{(\eps,u]}x\,h''(\dd x)
=uh'_{+}(u)-\eps h'_{+}(\eps)
-\int_{(\eps,u]}h'_{+}(x)\,\dd x.
\end{equation*}
Since $h$ is convex on $(0,a)$,
it is absolutely continuous
on compact intervals in $(0,a)$, hence
\begin{equation}
\label{eq:LemRA7}
\int_{(\eps,u]}x\,h''(\dd x)
=uh'_{+}(u)-\eps h'_{+}(\eps)
-h(u)+h(\eps).
\end{equation}
As $\eps\downarrow0$ we now get a contradiction
because the limit of the left-hand side of~\eqref{eq:LemRA7}
is finite due to~\eqref{eq:LemRA5},
while the limit of the right-hand side of~\eqref{eq:LemRA7}
equals~$\infty$ due to~\eqref{eq:LemRA6}
and $-\eps h'_{+}(\eps)\ge0$ for sufficiently small $\eps>0$.

It remains to prove the implication
\begin{equation*}
\text{\eqref{eq:LemRA1}--\eqref{eq:LemRA3}}
\Longrightarrow\text{\eqref{eq:LemRA4}},
\end{equation*}
which follows from Theorems~\ref{th:sem_dif1}
and~\ref{th:sem_dif3}, as it was observed above.
Such an argument is very indirect.
Let us now present a short direct argument.
Let $g$ satisfy \eqref{eq:LemRA1}--\eqref{eq:LemRA3}.
Clearly, \eqref{eq:LemRA7}~holds with $g$ instead of~$h$.
By~\eqref{eq:LemRA2} and~\eqref{eq:LemRA3},
there is a finite $\lim_{\eps\downarrow0}\eps g'_{+}(\eps)$.
Now using the integration by parts in a different way we obtain
\begin{equation}
\label{eq:LemRA8}
\int_{(\eps,u]}x(g'_{+}(x))^{2}\,\dd x
=\frac{(ug'_{+}(u))^{2}-(\eps g'_{+}(\eps))^{2}}2
-\int_{(\eps,u]}x^{2}g'_{+}(x)\,g''(\dd x).
\end{equation}
As $\eps\downarrow0$ the right-hand side,
hence also the left-hand side, of~\eqref{eq:LemRA8}
has a finite limit
(here~\eqref{eq:LemRA2} and the existence of a finite
$\lim_{\eps\downarrow0}\eps g'_{+}(\eps)$
are used).
Since $x(g'_{+}(x))^{2}$ is a nonnegative function,
statement~\eqref{eq:LemRA4}
follows by the monotone convergence
(or by Fatou's lemma).
\end{proof}

Theorems~\ref{th:sem_dif1} and~\ref{th:sem_dif2}
in the Brownian case
imply another result from real analysis,
which is again of interest in itself.

\begin{lemma}
\label{lem:LemRA2}
For some $a>0$, let $g\colon(0,a)\to\bbR$
be a convex or concave function
satisfying~\eqref{eq:LemRA3}.
Then, for any $u\in(0,a)$, it satisfies~\eqref{eq:LemRA2}.
\end{lemma}

Let us note that here assumption~\eqref{eq:LemRA3}
cannot be dropped: consider, for instance,
$g(x)=\frac1x$.
The way of proving Lemma~\ref{lem:LemRA2} via
Theorems~\ref{th:sem_dif1} and~\ref{th:sem_dif2}
is of course very indirect.
Let us present a direct proof.

\begin{proof}
In the first step let us establish that
$g'_{+}\in L^{1}_{\loc}(0+)$.
Since $g$ is convex or concave on $(0,a)$,
it is absolutely continuous
on compact intervals in $(0,a)$.
In particular, for $0<\eps<u<a$, we have
\begin{equation}
\label{eq:LemRA9}
\int_{(\eps,u]}g'_{+}(x)\,\dd x=g(u)-g(\eps).
\end{equation}
Again by convexity or concavity of~$g$,
$g'_{+}$~is monotone,
hence $g'_{+}$ is either nonnegative
or nonpositive in a sufficiently small
right neighborhood $(0,\delta)$ of zero.
Now $g'_{+}\in L^{1}_{\loc}(0+)$
follows from~\eqref{eq:LemRA9}
by letting $\eps\downarrow0$
and using the monotone convergence theorem
together with~\eqref{eq:LemRA3}.

Similarly to~\eqref{eq:LemRA7} we get
\begin{equation}
\label{eq:LemRA10}
\int_{(\eps,u]}x\,g''(\dd x)
=ug'_{+}(u)-\eps g'_{+}(\eps)-g(u)+g(\eps).
\end{equation}
Since $g$ is convex or concave, $g''(\dd x)$
is a positive or negative measure.
Therefore, the left-hand side,
hence also the right-hand side,
of~\eqref{eq:LemRA10}
has a finite or infinite limit
as $\eps\downarrow0$.
By~\eqref{eq:LemRA3},
there is a finite or infinite
$\lim_{\eps\downarrow0}\eps g'_{+}(\eps)$.
The latter limit can only be~$0$
(provided it exists)
because otherwise $g'_{+}\notin L^{1}_{\loc}(0+)$.
Hence
\begin{equation*}
\int_{(0,u]}x\,g''(\dd x)
=\lim_{\eps\downarrow0}
\int_{(\eps,u]}x\,g''(\dd x)
\text{ is finite}
\end{equation*}
(the equality holds by the monotone convergence).
We thus get~\eqref{eq:LemRA2}.
\end{proof}

%============================================
\newpar{ Another Characterisation of the Semimartingale Property.}
\label{par:CharBC}
Let $B$ be a Brownian motion starting from $x_{0}>0$.
In the following we consider the stopped process
$B^{\tau^{B}_{0}}$ with
$\tau^{B}_{0}:=\inf\{t\ge0:B_{t}=0\}$
and discuss the conditions 
on a Borel function
$g\colon[0,\infty)\to\bbR$, 
under which the process
$g(B^{\tau^{B}_{0}})$ is a semimartingale.
Under the assumption that $g$ is continuous at~$0$
and the restricted 
function $g|_{(0,\infty)}$
is a difference of two convex functions,
%on the interval $(0,\infty)$, 
a necessary and sufficient condition
is given in Theorem~\ref{th:sem_dif1} above.
Without any assumption,
a necessary and sufficient condition
is given in Theorem~5.9
in~\cite{CinlarJacodProtterSharpe:80}.
Here we enrich the picture in two ways:
firstly, we discuss the relations
between the elementary conditions
that form the necessary and sufficient condition
of Theorem~5.9 in~\cite{CinlarJacodProtterSharpe:80}
(namely, parts (ii) and~(iii) 
of~\cite[Th.~5.9]{CinlarJacodProtterSharpe:80}
imply part~(i) of that theorem);
secondly, we present another necessary and sufficient
condition for $g(B^{\tau^{B}_{0}})$
to be a semimartingale.

In order to formulate the result
we introduce several conditions:
\begin{gather}
\label{eq:CharBC1}
\text{the restriction }
g|_{(0,\infty)}
\text{ is a difference of two convex functions }
(0,\infty)\to\bbR,\\
\label{eq:CharBC2}
\text{there exists a finite limit }
g(0):=\lim_{x\downarrow0}g(x),\\
\label{eq:CharBC3}
x\in L^{1}_{\loc}(0+,|g''|(\dd x)),\\
\label{eq:CharBC4}
g=h_{1}-h_{2}
\text{ with }
h_{i}\colon[0,\infty)\to\bbR
\text{ convex and continuous at }
0,\quad i=1,2.
\end{gather}

\begin{remark}
\label{rem:CharBC1}
Let us note that condition~\eqref{eq:CharBC4}
is strictly stronger than
\eqref{eq:CharBC1} and~\eqref{eq:CharBC2}.
For instance, the functions $g$
constructed in Examples~\ref{ex:sem_ex1}
and~\ref{ex:sem_ex2} satisfy~\eqref{eq:CharBC1}
and~\eqref{eq:CharBC2},
but for them, $g(B^{\tau^{B}_{0}})$
is not a semimartingale, hence,
by Theorem~\ref{th:CharBC1} below,
\eqref{eq:CharBC4}~fails.
\end{remark}

\begin{theorem}
\label{th:CharBC1}
Let $g\colon[0,\infty)\to\bbR$ be a Borel function.
The following are equivalent:

(a) $g(B^{\tau^{B}_{0}})$ is a semimartingale;

(b) \eqref{eq:CharBC1} and~\eqref{eq:CharBC3} hold;

(c) \eqref{eq:CharBC4} holds.
\end{theorem}

\begin{proof}
If \eqref{eq:CharBC4} holds, then,
by Corollary~\ref{cor:sem_dif1},
$h_{i}(B^{\tau^{B}_{0}})$
are semimartingales, $i=1,2$
(alternatively, one can use Lemma~\ref{lem:LemRA2} here).
Thus, $\text{(c)}\Rightarrow\text{(a)}$.
By~\cite[Th.~5.9]{CinlarJacodProtterSharpe:80},
(a)~is equivalent to \eqref{eq:CharBC1}--\eqref{eq:CharBC3}.
In particular, $\text{(a)}\Rightarrow\text{(b)}$.

It remains to prove that $\text{(b)}\Rightarrow\text{(c)}$.
Assume~\eqref{eq:CharBC1} and~\eqref{eq:CharBC3}.
By Lemma~\ref{lem:LemRA1}, \eqref{eq:CharBC2}~holds. Let
\begin{equation*}
g''(\dd x)=\nu_{1}(\dd x)-\nu_{2}(\dd x)
\end{equation*}
be the Jordan decomposition of the locally finite signed
measure $g''(\dd x)$ on $(0,\infty)$,
that is $\nu_{i}(\dd x)$ are locally finite
positive measures on $(0,\infty)$
such that $\nu_{1}\perp\nu_{2}$.
In particular, we have
\begin{equation*}
\nu_{1}(\dd x)+\nu_{2}(\dd x)=|g''|(\dd x),
\end{equation*}
hence
\begin{equation}
\label{eq:CharBC5}
x\in L^{1}_{\loc}(0+,\nu_{i}),\quad i=1,2.
\end{equation}
For $i=1,2$, define the functions
\begin{equation*}
k_{i}(x)=\begin{cases}
\nu_{i}((1,x])&\text{if }x\in[1,\infty),\\
-\nu_{i}((x,1])&\text{if }x\in(0,1).
\end{cases}
\end{equation*}
Let us prove that~\eqref{eq:CharBC4}
is satisfied with functions $h_{i}=H_{i}$,
where
\begin{equation*}
H_{i}(x)=\int_{1}^{x}k_{i}(y)\,\dd y+a_{i}x+b_{i},
\quad x\in[0,\infty),\quad i=1,2,
\end{equation*}
for a suitable choice of constants $a_{i},b_{i}$.
Since $k_{i}$ are nondecreasing and right-continuous
and $(H_{i})'_{+}=k_{i}+a_{i}$, we have that
$H_{i}$ are convex functions on $(0,\infty)$.
By construction it holds
\begin{equation*}
(H_{1}-H_{2})'_{+}(x)
=g'_{+}(x)-g'_{+}(1)+a_{1}-a_{2},
\quad x\in(0,\infty).
\end{equation*}
Choosing $a_{1}$ and $a_{2}$ so that
$a_{1}-a_{2}=g'_{+}(1)$,
$b_{1}$ and $b_{2}$ so that
$(H_{1}-H_{2})(1)=g(1)$,
we obtain that $g=H_{1}-H_{2}$ on $(0,\infty)$.
It remains to prove that
$\lim_{x\downarrow0}H_{i}(x)<\infty$, $i=1,2$.
To this end, it is enough to prove that
$\int_{0}^{1}\nu_{i}((y,1])\,\dd y<\infty$.
For $i=1,2$, we have
\begin{equation*}
\int_{0}^{1}\nu_{i}((y,1])\,\dd y
=\int_{(0,1]}\int_{(0,1]}I(y<x\le1)\,\nu_{i}(\dd x)\,\dd y
=\int_{(0,1]}x\,\nu_{i}(\dd x)<\infty
\end{equation*}
by~\eqref{eq:CharBC5}.
This concludes the proof.
\end{proof}

%============================================
\section{Finiteness of Additive Functionals of Diffusion Processes}
\label{sec:Additive_F}
In this section we study
the finiteness of the process
\begin{equation}
\label{eq:conv loc time nu}
\int_J L^y_t(Y)\,\nu(\dd y),\quad t\in[0,\zeta],
\end{equation}
where 
$\nu$
is an arbitrary positive measure defined on the Borel 
$\sigma$-field
$\cB(J)$
(setting and notations in Section~\ref{sec:set} apply),
and $(L_{t}^{y}(Y);t\in[0,\zeta),y\in J)$
is an a.s. continuous in $t$ and c\`adl\`ag in $y$
version of the local time of~$Y$
(in fact, it will be even a.s. jointly continuous
in $(t,y)$; see~\cite[Prop.~A.1]{MijatovicUrusov:12a}).
The characterisation of the finiteness
of the additive functional
given in~\eqref{eq:conv loc time nu}
plays a key role in the proofs of the results of Section~\ref{sec:sem_dif}.
The occupation times formula
(see \cite[Ch.~VI, Cor.~1.6]{RevuzYor:99})
implies that this question has been answered in~\cite{MijatovicUrusov:12c}
in the case the measure 
$\nu$
is absolutely continuous with respect to the 
Lebesgue measure
$\nu_L$.
In this section we 
give a deterministic characterisation of the finiteness
of the additive functional in~\eqref{eq:conv loc time nu}
for a general positive (possibly non-locally finite) measure
$\nu$
on the interval~$J$.

We proceed in two steps.
First we reduce the study of the finiteness
of~\eqref{eq:conv loc time nu}
in general to the question of the convergence of the integral
\begin{equation}
\label{eq:loc time at exit}
\int_J L^y_\zeta(Y)\,\nu(\dd y),
\end{equation}
where the measure $\nu$ is now locally finite on~$J$.
In the second step we formulate the answer to 
the latter problem in terms of a
deterministic integrability criterion
involving the scale function
$s$ and its derivative $\rho$,
given in \eqref{eq:set4}--\eqref{eq:set5},
and the measure~$\nu$.

Let us consider a general positive measure $\nu$
on~$J$. With $B_{\eps}(x):=(x-\eps,x+\eps)$ we set
\begin{equation*}
D^{\nu}:=\{l,r\}\cup\{x\in J:
\forall\eps>0\text{ it holds }\nu(B_{\eps}(x))=\infty\},
\end{equation*}
i.e. $D^{\nu}$ is the set of points in~$J$,
where the local finiteness of $\nu$ fails,
augmented with $\{l,r\}$.
Clearly, $D^{\nu}$ is closed in~$\ol J$.
For a closed subset $E$ in~$\ol J$
and $a,b\in\ol J$, let us define the stopping times
\begin{align*}
\tau^{Y}_{E}&:=\inf\{t\in[0,\infty):Y_{t}\in E\}
\quad(\inf\emptyset:=\infty),\\
\tau^{Y}_{a}&:=\tau^{Y}_{\{a\}},\\
\tau^{Y}_{a,b}&:=\tau^{Y}_{a}\wedge\tau^{Y}_{b}.
\end{align*}
We start with the following result.

\begin{theorem}
\label{thm:Local_Time_nu}
$\PP$-a.s. we have:
\begin{align}
\label{eq:before_nu_explodes}
\int_J L^y_t(Y)\,\nu(\dd y)&<\infty,\quad t\in[0,\tau^Y_{D^\nu}),\\
\label{eq:after_nu_explodes}
\int_J L^y_t(Y)\,\nu(\dd y)&=\infty,\quad t\in(\tau^Y_{D^\nu},\zeta].
\end{align}
\end{theorem}

\begin{remark}
\label{rem:Local_Time_nu}
Once Theorem~\ref{thm:Local_Time_nu} 
is established,
it remains to study the convergence of the integral
\begin{equation*}
\int_J L^y_{\tau^Y_{D^\nu}}(Y)\,\nu(\dd y).
\end{equation*}
If $x_0\in {D^\nu}$, then there is nothing to study here
because $\tau^Y_{D^\nu}\equiv0$
and $\int_J L^y_0(Y)\,\nu(\dd y)=0$.
Assume now that $x_0\notin D^\nu$ and define
\begin{equation}
\label{eq:add fun al be}
\alpha=\sup([l,x_0)\cap D^\nu)
\quad\text{and}\quad
\beta=\inf((x_0,r]\cap D^\nu).
\end{equation}
Then we have
$\tau^Y_{D^\nu}=\tau^Y_{\alpha,\beta}$.
Now if we consider $I:=(\alpha,\beta)$ as a new state space
for~$Y$,
then $\tau^Y_{\alpha,\beta}$ will be the new exit time,
and we clearly have that $\nu$
is locally finite on~$I$.
This concludes the reduction of the study of
the finiteness of the process
in~\eqref{eq:conv loc time nu},
with a general positive measure~$\nu$,
to the question of the convergence of the integral
given in~\eqref{eq:loc time at exit}
with measure~$\nu$, which is now locally finite on $J$. 
\end{remark}

\begin{proof}[Proof of Theorem~\ref{thm:Local_Time_nu}]
If $x_{0}\in D^{\nu}$, then there is nothing to prove
in~\eqref{eq:before_nu_explodes}. Let $x_{0}\notin D^{\nu}$.
A.s. on $\{t<\tau^Y_{D^\nu}\}$
the following holds:
$[\inf_{u\le t}Y_{u},\sup_{u\le t}Y_{u}]\subset(\alpha,\beta)$
with $\alpha$ and $\beta$ from~\eqref{eq:add fun al be},
hence
$\nu\left([\inf_{u\le t}Y_u,\sup_{u\le t}Y_u]\right)<\infty$,
and the function $y\mapsto L^y_t(Y)$ is bounded
as a c\`adl\`ag function with a compact support.
Thus, statement~\eqref{eq:before_nu_explodes} follows.

As for~\eqref{eq:after_nu_explodes},
let us first assume that $x_{0}\notin D^{\nu}$.
Then $\tau^{Y}_{D^{\nu}}=\tau^{Y}_{\alpha,\beta}$,
hence $\{\tau^{Y}_{D^{\nu}}<t<\zeta\}
=\{\tau^{Y}_{\alpha}<t<\zeta\}
\cup\{\tau^{Y}_{\beta}<t<\zeta\}$.
If $\PP(\tau^{Y}_{\alpha}<t<\zeta)>0$
(in particular, this means that $\alpha>l$),
then~\eqref{eq:after_nu_explodes} holds
a.s. on $\{\tau^{Y}_{\alpha}<t<\zeta\}$
because $\alpha\in J\cap D^{\nu}$ and,
by~\cite[Th.~2.7]{ChernyEngelbert:05},
the function $y\mapsto L_{t}^{y}(Y)$
is strictly positive in some neighbourhood
of $\alpha$ a.s. on $\{\tau^{Y}_{\alpha}<t<\zeta\}$.
Similarly, \eqref{eq:after_nu_explodes} holds
a.s. on $\{\tau^{Y}_{\beta}<t<\zeta\}$.
In the case $x_{0}\in D^{\nu}$
statement~\eqref{eq:after_nu_explodes}
again follows
from~\cite[Th.~2.7]{ChernyEngelbert:05}
by the same reasoning.
\end{proof}

It now remains to study the convergence of the integral
in~\eqref{eq:loc time at exit}
under the assumption that the measure
$\nu$ on $J$ is locally finite.
The answer depends on the behaviour of~$Y$.
Theorems~\ref{thm:Loc_time_int_nu_exit_A}
and~\ref{thm:Loc_time_int_nu_exit_BC} below
examine the cases
$\PP(A)=1$ and $\PP(B_r\cup B_l\cup C_r\cup C_l)=1$
separately
(the events $A$, $B_{r}$, $B_{l}$, $C_{r}$, $C_{l}$
are defined in Appendix~\ref{app:Friff_Prop};
see Propositions~\ref{prop:set2} and~\ref{prop:set3}
for the description of these cases).

\begin{theorem}
\label{thm:Loc_time_int_nu_exit_A}
Let $\nu$
be a locally finite positive measure on the interval
$J=(l,r)$.
Assume that $s(r)=\infty$ and $s(l)=-\infty$.
Then $\PP$-a.s. we have
\begin{equation}
\label{eq:Loc_time_diff_infinite}
L^y_\zeta(Y)=\infty
\text{ for every }y\in J,
\end{equation}
hence
$\int_J L^y_\zeta(Y)\,\nu(\dd y)=\infty$
$\PP$-a.s.
whenever $\nu$ is a non-zero measure
(i.e.~$\nu(J)>0$).
\end{theorem}

Let us remark that the assumption
$s(r)=\infty$ and $s(l)=-\infty$
of Theorem~\ref{thm:Loc_time_int_nu_exit_A}
is equivalent to $\PP(A)=1$
(see Propositions~\ref{prop:set2} and~\ref{prop:set3}).
In particular,
in Theorem~\ref{thm:Loc_time_int_nu_exit_A}
we have $\zeta=\infty$~$\PP$-a.s.

The study of the remaining case
$\PP(B_{r}\cup B_{l}\cup C_{r}\cup C_{l})=1$
consists of the investigation
of the convergence of~\eqref{eq:loc time at exit}
on the event $\{\lim_{t\uparrow\zeta}Y_t=l\}$
and on the event $\{\lim_{t\uparrow\zeta}Y_t=r\}$.
In the following theorem we investigate the convergence
of~\eqref{eq:loc time at exit}
on the event $\{\lim_{t\uparrow\zeta}Y_t=l\}$
(in particular, we need to assume $s(l)>-\infty$,
which is, by Proposition~\ref{prop:set3},
equivalent to $\PP(\lim_{t\uparrow\zeta}Y_t=l)>0$).

\begin{theorem}
\label{thm:Loc_time_int_nu_exit_BC}
Let $\nu$
be a locally finite positive measure on the interval
$J=(l,r)$.
Assume that $s(l)>-\infty$.

(i) If
\begin{equation*}
\frac{s-s(l)}{\rho}\in L^1_\loc(l+,\nu),
\end{equation*}
then
\begin{equation*}
\int_J L^y_\zeta(Y)\,\nu(\dd y) <\infty\quad\PP\text{-a.s. on }\left\{\lim_{t\uparrow\zeta}Y_t=l\right\}.
\end{equation*}

(ii) If
\begin{equation*}
\frac{s-s(l)}{\rho}\notin L^1_\loc(l+,\nu),
\end{equation*}
then
\begin{equation*}
\int_J L^y_\zeta(Y)\,\nu(\dd y) =\infty\quad\PP\text{-a.s. on }\left\{\lim_{t\uparrow\zeta}Y_t=l\right\}.
\end{equation*}
\end{theorem}

The investigation of the convergence
of~\eqref{eq:loc time at exit} on the event
$\{\lim_{t\uparrow\zeta}Y_t=r\}$ is similar.
This completes the study of the convergence
of the integral in~\eqref{eq:loc time at exit}.

\begin{proof}[Proofs
of Theorems~\ref{thm:Loc_time_int_nu_exit_A}
and~\ref{thm:Loc_time_int_nu_exit_BC}]
It is clear that Theorem~\ref{thm:Loc_time_int_nu_exit_A}
follows if we prove the equality
in~\eqref{eq:Loc_time_diff_infinite}.
By the Dambis-Dubins-Schwarz theorem,
there exists a Brownian motion $B$ starting from $s(x_0)$
(possibly on an enlargement of the initial probability space) such that
\begin{equation}
\label{eq:pt3}
s(Y_t)=B_{\la s(Y), s(Y)\ra_t}\quad\PP\text{-a.s.},\quad t\in[0,\zeta).
\end{equation}
Since $s(r)=\infty$ and $s(l)=-\infty$, $\PP$-a.s. we have
$\limsup_{t\uparrow\zeta}s(Y_t)=\infty$,
$\liminf_{t\uparrow\zeta}s(Y_t)=-\infty$,
hence $\la s(Y), s(Y)\ra_\zeta=\infty$~$\PP$-a.s.
It can be deduced from the It\^o-Tanaka formula
that $\PP$-a.s. it holds
\begin{equation}
\label{eq:Loc_Time_equality}
L_t^{y}(Y)
=\frac{1}{\rho(y)}
L_{\la s(Y), s(Y)\ra_t}^{s(y)}(B),
\quad(t,y)\in[0,\zeta)\times J.
\end{equation}
Since $\PP$-a.s. we have
$L^z_\infty(B)=\infty$
for any
$z\in\bbR$
(see e.g. \cite[Ch.~VI, \S~2]{RevuzYor:99}),
the equality in~\eqref{eq:Loc_time_diff_infinite} and 
Theorem~\ref{thm:Loc_time_int_nu_exit_A} follow.

We prove Theorem~\ref{thm:Loc_time_int_nu_exit_BC}
by reducing it to Lemma~\ref{lem:BM_Loc_Fuct} below,
which deals with an analogous problem for a Brownian motion.
Note first that~\eqref{eq:Loc_Time_equality} 
implies the following equality
\begin{equation}
\label{eq:main_loc_time_int_eqaulity}
\int_J L^y_\zeta(Y)\,\nu(\dd y) =
 \int_JL^{s(y)}_{\la s(Y), s(Y)\ra_\zeta}(B)\,\frac{\nu(\dd y)}{\rho(y)} \quad\pn{\PP} 
\end{equation}
Since $s(l)>-\infty$, we have $\PP(L)>0$, where
$L:=\{\lim_{t\uparrow\zeta} Y_t=l\}$.
By the equality in~\eqref{eq:pt3} it follows that
$\lim_{t\uparrow\zeta}B_{\la s(Y), s(Y)\ra_t}=s(l)$
$\PP$-a.s.~on~$L$ and hence
\begin{equation}
\label{eq:tau_time_change_at_exit}
\la s(Y), s(Y)\ra_\zeta=\tau^B_{s(l)}
\quad\PP\text{-a.s. on }L,
\end{equation}
where
$\tau^B_{s(l)}$
is the first time the Brownian motion 
$B$
hits the level
$s(l)$.
Define
$\wt\nu(\dd y):=\nu(\dd y)/\rho(y)$,
$y\in J$,
and let
$\wt\mu$
be the pushforward measure of
$\wt\nu$
via
$s$:
$\wt \mu(E)=\wt \nu(s^{-1}(E))$
for any Borel subset 
$E\subseteq s(J)$.
Equalities~\eqref{eq:main_loc_time_int_eqaulity}
and~\eqref{eq:tau_time_change_at_exit}
yield
\begin{equation*}
\int_J L^y_\zeta(Y)\,\nu(\dd y)
=\int_{s(J)}L^{z}_{\tau^B_{s(l)}}(B)\,\wt\mu(\dd z)
\quad\PP\text{-a.s. on }L.
\end{equation*}
Theorem~\ref{thm:Loc_time_int_nu_exit_BC}
now follows from 
\begin{equation*}
\int_{(s(l),s(z))}(x-s(l))\,\wt\mu(\dd x)
=\int_{(l,z)}\frac{s(y)-s(l)}{\rho(y)}\,\nu(\dd y),
\quad z\in J,
\end{equation*}
and an application of Lemma~\ref{lem:BM_Loc_Fuct}.
\end{proof}

\begin{lemma}
\label{lem:BM_Loc_Fuct}
For some $l\in\bbR$, define $I:=(l,\infty)$.
Let $B$ be a Brownian motion starting from $x_0\in I$
and $\nu$ a locally finite positive measure on~$I$.
Let $\tau^B_l$ denote the first time
$B$ hits the level~$l$.

(i) If $x-l\in L^1_\loc(l+,\nu)$, then
\begin{equation*}
\int_I L^y_{\tau^B_l}(B)\,\nu(\dd y)<\infty\quad\PP\text{-a.s.}
\end{equation*}

(ii) If $x-l\notin L^1_\loc(l+,\nu)$, then
\begin{equation*}
\int_I L^y_{\tau^B_l}(B)\,\nu(\dd y)=\infty\quad\PP\text{-a.s.}
\end{equation*}
\end{lemma}

\begin{remark}
Lemma~\ref{lem:BM_Loc_Fuct} is known and has a long history.
On the one hand, Lemma~\ref{lem:BM_Loc_Fuct} contains
in itself Lemma~5.10 in~\cite{CinlarJacodProtterSharpe:80},
which is complemented by a criterion
for a.s.-infiniteness of the additive functional.
That is why the proof below is different from the one
of Lemma~5.10 in~\cite{CinlarJacodProtterSharpe:80}.
On the other hand, Lemma~\ref{lem:BM_Loc_Fuct}
appeared in the literature already in this form.
It can be traced back to~\cite[Lem.~1.4.1]{Assing:94}
(the discussion in~\cite[Sec.~4]{MijatovicUrusov:12c}
gives a detailed account of the history of this result). 
The proof in~\cite[Lem.~1.4.1]{Assing:94}
is based on the Ray-Knight theorem
and an application of Jeulin's~\cite{Jeulin:80} lemma
(e.g.~\cite[Lem.~1.4.2]{Assing:94}).
Here we give a proof, which replaces
the application of Jeulin's lemma by a 
simple direct argument.
%(cf.~\cite[Sec.~6]{MijatovicUrusov:12c}).
\end{remark}

\begin{proof}
The mapping $x\mapsto L^x_{\tau^B_l}(B)$ is  
$\PP$-a.s. 
a continuous function with compact support in
$[l,\infty)$.
Therefore the finiteness of 
the integral 
$\int_I L^y_{\tau^B_l}(B)\,\nu(\dd y)$
reduces to the question
\begin{equation*}
\text{whether}\quad
\int_{(l,x_0)}  L^y_{\tau^B_l}(B)\,\nu(\dd y)=
\int_{(0,x_0-l)}  L^{l+u}_{\tau^B_l}(B)\,\nu(l+\dd u)
\quad\text{is finite.}
\end{equation*}

Let $\ol W$ and $\wt W$ be independent Brownian motions starting from~0.
Let us set
$\eta_t=\ol W_t^2+\wt W_t^2$,
i.e. 
$\eta=(\eta_t)_{t\in[0,\infty)}$ 
is a squared two-dimensional Bessel process starting from~0.
It follows from the first Ray-Knight theorem that
\begin{equation*}
\Law\left(L^{l+u}_{\tau^B_l}(B);u\in[0,x_0-l]\right)
=\Law\left(\eta_u;u\in[0,x_0-l]\right).
\end{equation*}
Therefore, the question is
\begin{equation}
\label{eq:Law_Loc_Time}
\text{whether}\quad
\int_{(0,x_{0}-l)}\eta_{u}\,\nu(l+\dd u)
=\int_{(l,x_{0})}\eta_{y-l}\,\nu(\dd y)
\quad\text{is finite.}
\end{equation}
In what follows we prove that, for a Brownian motion $W$ starting from~0,

\smallskip\noindent
(A) $x-l\in L^1_\loc(l+,\nu)$ implies that
$\int_{(l,x_{0})}W_{y-l}^2\,\nu(\dd y)<\infty$
$\PP$-a.s.;

\smallskip\noindent
(B) $x-l\notin L^1_\loc(l+,\nu)$ implies that
$\int_{(l,x_{0})}W_{y-l}^2\,\nu(\dd y)=\infty$
$\PP$-a.s.

\noindent
Together with~\eqref{eq:Law_Loc_Time} this will complete the proof of Lemma~\ref{lem:BM_Loc_Fuct}.

By Fubini's theorem we have
$\EE\int_{(l,x_{0})}W_{y-l}^2\,\nu(\dd y)
=\int_{(l,x_{0})}(y-l)\,\nu(\dd y)$
and (A) follows.

In order to prove (B) we assume that
\begin{equation}
\label{eq:nu_finite}
\PP\left(\int_{(l,x_{0})}
W_{y-l}^2\,\nu(\dd y)<\infty\right)>0.
\end{equation}
Then there exists a large $M<\infty$ such that
$\gamma:=\PP(E)>0$,
where
$$
E:=\left\{\int_{(l,x_{0})}
W_{y-l}^2\,\nu(\dd y)\le M\right\}.
$$
For any positive $\delta$ and $u$, the probability
$\PP(W_u^2\ge\delta^2u)=\PP(|N(0,1)|\ge\delta)$
does not depend on~$u$.
Pick a sufficiently small $\delta>0$ such that
$\PP(|N(0,1)|\ge\delta)\ge1-\frac\gamma2$
and note that, for any~$y\in(l,x_{0})$,
we have
$$
\EE\left(W_{y-l}^2I_E\right)
\ge\EE\left(W^{2}_{y-l}
I_{E\cap\{W^{2}_{y-l}\ge\delta^{2}(y-l)\}}\right)
\ge\frac{\delta^2\gamma}2(y-l).
$$
By Fubini's theorem,
\begin{equation*}
M\ge\EE\left[I_E\int_{(l,x_{0})}
W_{y-l}^2\,\nu(\dd y)\right]
=\int_{(l,x_{0})}
\EE(W_{y-l}^2I_E)\,\nu(\dd y)
\ge\frac{\delta^2\gamma}2
\int_{(l,x_{0})}(y-l)\,\nu(\dd y).
\end{equation*}
Hence~\eqref{eq:nu_finite} 
implies $x-l\in L^1_\loc(l+,\nu)$, which proves~(B),
and the lemma follows.
\end{proof}

%============================================
\section{Proofs of Theorems from Section~\protect\ref{sec:sem_dif}}
\label{sec:sem_p}
In this section we will prove Theorems \ref{th:sem_dif1}, \ref{th:sem_dif2}, \ref{th:sem_dif3}
and~\ref{th:sem_dif4}. Let us assume~\eqref{eq:sem_dif1} and case~(B) of Section~\ref{sec:sem_dif}.

\newitem
\label{it:sem_p1}
Consider a sequence $(\alpha_n)_{n\in\bbN}$, $l<\alpha_n<x_0$, $\alpha_n\downarrow l$.
By the It\^o-Tanaka formula
applied to the stopped process
$g(Y^{\tau^Y_{\alpha_n}})$, $n\in\bbN$, we get that $\PP$-a.s. it holds:
\begin{equation}
\label{eq:sem_p1_1}
g(Y_t)=g(x_0)+\ol A_t+\ol M_t,\quad t\in[0,\zeta),
\end{equation}
where the locally finite measure
$\nu_g$
on the interval $J$
is defined in~\eqref{eq:nu_g}
and 
\begin{align*}
\ol A_t&=\int_J L_t^y(Y)\,\nu_g(\dd y),\quad t\in[0,\zeta),\\
\ol M_t&=\int_0^t(g'\sigma)(Y_u)\,\dd W_u,\quad t\in[0,\zeta).
\end{align*}
Let us note that
the process $\ol M=(\ol M_t)_{t\in[0,\zeta)}$
is a continuous local martingale
on the stochastic interval $[0,\zeta)$ with
\begin{equation}
\label{eq:Quad_Var_loc_time_mart}
\la\ol M,\ol M\ra_t=\int_0^t(g'\sigma)^2(Y_u)\,\dd u
=\int_J L_t^y(Y)(g')^2(y)\,\dd y,\quad t\in[0,\zeta)
\end{equation}
(the second equality follows from the occupation times formula),
and the process $\ol A=(\ol A_t)_{t\in[0,\zeta)}$
has a locally finite variation on $[0,\zeta)$.

Denote by $\Var\ol A=(\Var\ol A_t)_{t\in[0,\zeta)}$
the variation process of $\ol A$.
$\PP$-a.s. it holds that
\begin{equation}
\label{eq:sem_p1_2}
\Var\ol A_t=\int_J L_t^y(Y)\,|\nu_g|(\dd y),
\quad t\in[0,\zeta),
\end{equation}
where $|\nu_g|$ is the variation measure of~$\nu_g$.
We will now prove~\eqref{eq:sem_p1_2} by a pathwise argument,
but let us first observe that the right-hand side
of~\eqref{eq:sem_p1_2}
is, clearly, $(\cF_{t})$-adapted and finite;
finiteness $\PP$-a.s. on $\{t<\zeta\}$
follows from the fact that
$\PP$-a.s. on $\{t<\zeta\}$
the function $y\mapsto L_t^y(Y)$
is c\`adl\`ag with a compact support
and the measure $|\nu_g|$ is locally finite on~$J$.
To prove~\eqref{eq:sem_p1_2}, note that
$\PP$-a.s. on $\{t<\zeta\}$
there exists a compact interval 
$I\subset J$, which depends on $\omega$
and contains the support of $y\mapsto L_t^y(Y)$.
Let $\omega$ be fixed.
Since
$|\nu_g|(I)<\infty$,
there exists a Jordan decomposition
$\nu_g=\nu_g^+-\nu_g^-$:
$\nu_g^+$
and
$\nu_g^-$
are positive measures 
and 
$\nu_g^+(\cdot)=\nu_g(\cdot\cap P)$
and
$\nu_g^-(\cdot)=-\nu_g(\cdot\cap (I\setminus P))$
for some Borel set 
$P$
in
$I$.
Furthermore,
on
$I$
it holds
$|\nu_g|=\nu_g^++\nu_g^-$.
Note that 
\begin{equation}
\label{eq:A_fin_var_decomp}
\ol A_t(\omega)
=\int_I L_t^y(Y)(\omega)\,\nu^+_g(\dd y)
-\int_I L_t^y(Y)(\omega)\,\nu^-_g(\dd y),
\quad t\in[0,\zeta(\omega)),
\end{equation}
is a decomposition of 
$\ol A(\omega)$
into a difference of two non-decreasing
continuous functions.
To show~\eqref{eq:sem_p1_2},
it is sufficient to prove that
the measures on
$[0,\zeta(\omega))$
induced by these functions,
i.e. the measures
\begin{equation}
\label{eq:sign_measurs_on_[0t]}
\int_I \dd L_s^y(Y)(\omega)\,\nu^+_g(\dd y)
\quad\text{and}\quad
\int_I \dd L_s^y(Y)(\omega)\,\nu^-_g(\dd y),
\end{equation}
are singular
(that is the decomposition in~\eqref{eq:A_fin_var_decomp}
is minimal).
It is in fact easy to see that 
the former measure
is concentrated on the set
\begin{equation*}
\wt P=\{u\in[0,t]:Y_u(\omega)\in P\},
\end{equation*}
while the latter measure is concentrated
on the similar set, where $P$ is replaced by $I\setminus P$.
Indeed, by Fubini's theorem we have
\begin{equation*}
\int_{[0,\zeta(\omega))}I_{\wt P}(s)
\int_I \dd L_s^y(Y)(\omega)\,\nu^-_g(\dd y)
=\int_I \left(
\int_{[0,\zeta(\omega))}I_{\wt P}(s)\,\dd L_s^y(Y)(\omega)
\right)\nu^-_g(\dd y)=0,
\end{equation*}
and a similar argument applies for the other statement.
Thus, \eqref{eq:sem_p1_2}~follows.

%============================================
\newitem
\label{it:sem_p2}
Whenever
\begin{equation}
\label{eq:sem_p2_1}
\PP\text{-a.s. on }\{\zeta<\infty\}\text{ there exists a finite limit }\ol M_\zeta:=\lim_{t\uparrow\zeta}\ol M_t,
\end{equation}
we extend the process $(\ol M_t)_{t\in[0,\zeta)}$ to the process $M=(M_t)_{t\in[0,\infty)}$ by setting
\begin{equation}
\label{eq:sem_p2_2}
M_t:=\ol M_{t\wedge\zeta},\quad t\in[0,\infty).
\end{equation}
Let us prove that under~\eqref{eq:sem_p2_1} $M$ is a local martingale (now on the whole $[0,\infty)$).
Indeed, there exists a sequence of stopping times $(\eta_n)_{n\in\bbN}$ such that $\eta_n\uparrow\zeta$~$\PP$-a.s.
and $M^{\eta_n}$ is a martingale for any $n\in\bbN$. For $m\in\bbN$, set
\begin{equation*}
\xi_m=\inf\{t\in[0,\infty)\colon|M_t|\ge m\}\qquad(\inf\emptyset:=\infty)
\end{equation*}
and note that $\xi_m\uparrow\infty$~$\PP$-a.s. as $m\uparrow\infty$.
Since, for a fixed $m\in\bbN$, the processes $M^{\eta_n\wedge\xi_m}$, $n\in\bbN$,
are uniformly (in~$n$) bounded martingales and $M^{\eta_n\wedge\xi_m}_t\to M^{\xi_m}_t$~$\PP$-a.s. as $n\to\infty$
(note that $M$ is stopped at~$\zeta$), then the process $M^{\xi_m}$ is a martingale for any $m\in\bbN$.
Thus, $M=(M_t)_{t\in[0,\infty)}$ is a local martingale.

%============================================
\newitem
\label{it:sem_p3}
Since we consider case~(B) of
Section~\ref{sec:sem_dif}, we have
$\lim_{t\uparrow\zeta}Y_t=l$
$\PP$-a.s. on $\{\zeta<\infty\}$,
and there is a finite limit $g(l):=\lim_{x\downarrow l}g(x)$.
Then it follows from~\eqref{eq:sem_p1_1} that condition~\eqref{eq:sem_p2_1} is equivalent to
\begin{equation}
\label{eq:sem_p3_1}
\PP\text{-a.s. on }\{\zeta<\infty\}\text{ there exists a finite limit }\ol A_\zeta:=\lim_{t\uparrow\zeta}\ol A_t.
\end{equation}
Whenever \eqref{eq:sem_p3_1} holds, we extend the process $(\ol A_t)_{t\in[0,\zeta)}$ to the process $A=(A_t)_{t\in[0,\infty)}$ by setting
\begin{equation}
\label{eq:sem_p3_1a}
A_t:=\ol A_{t\wedge\zeta},\quad t\in[0,\infty).
\end{equation}
Finally, let us note that the condition
\begin{equation}
\label{eq:sem_p3_2}
\Var\ol A_\zeta<\infty\quad\PP\text{-a.s. on }\{\zeta<\infty\}
\end{equation}
implies~\eqref{eq:sem_p3_1} and under~\eqref{eq:sem_p3_2} the process $A=(A_t)_{t\in[0,\infty)}$
has a locally finite variation (on the whole $[0,\infty)$).

%============================================
\newitem
\label{it:sem_p4}
By applying Theorem~\ref{thm:Loc_time_int_nu_exit_BC}
with the positive measure
$\nu(\dd y) = (g')^2(y)\dd y$,
we obtain from~\eqref{eq:Quad_Var_loc_time_mart}
the following alternative
(additionally use the Dambis-Dubins-Schwarz theorem
for continuous local martingales on stochastic intervals):

\smallskip\noindent
$(M_1)$ If \eqref{eq:sem_dif7} is satisfied, then
\begin{equation*}
\la\ol M,\ol M\ra_\zeta<\infty\quad\PP\text{-a.s. on }\{\zeta<\infty\},
\end{equation*}
hence \eqref{eq:sem_p2_1} and~\eqref{eq:sem_p3_1} hold.

\smallskip\noindent
$(M_2)$ If \eqref{eq:sem_dif7} is violated, then
\begin{equation*}
\la\ol M,\ol M\ra_\zeta=\infty\quad\PP\text{-a.s. on }\{\zeta<\infty\},
\end{equation*}
hence
\begin{align}
\label{eq:sem_p4_1}
\limsup_{t\uparrow\zeta}\ol M_t
&=-\liminf_{t\uparrow\zeta}\ol M_t=\infty\quad\PP\text{-a.s. on }\{\zeta<\infty\},\\
\label{eq:sem_p4_2} 
\limsup_{t\uparrow\zeta}\ol A_t
&=-\liminf_{t\uparrow\zeta}\ol A_t=\infty\quad\PP\text{-a.s. on }\{\zeta<\infty\}.
\end{align}
(Let us note that \eqref{eq:sem_p4_2} follows from~\eqref{eq:sem_p4_1}
via~\eqref{eq:sem_p1_1}.)
Applying
Theorem~\ref{thm:Loc_time_int_nu_exit_BC}
once again 
with the measure
$\nu=|\nu_g|$,
we get from~\eqref{eq:sem_p1_2} another alternative:

\smallskip\noindent
$(A_1)$ \eqref{eq:sem_dif3} implies~\eqref{eq:sem_p3_2}.

\smallskip\noindent
$(A_2)$ If \eqref{eq:sem_dif3} is violated, then
\begin{equation*}
\Var\ol A_\zeta=\infty\quad\PP\text{-a.s. on }\{\zeta<\infty\}.
\end{equation*}

%============================================
\newitem
\label{it:sem_p5}
Let us now assume that $g(Y)$ is a semimartingale, i.e.
\begin{equation*}
g(Y_t)=g(x_0)+\wt A_t+\wt M_t,\quad t\in[0,\infty),
\end{equation*}
with a continuous process $\wt A=(\wt A_t)_{t\in[0,\infty)}$
of a locally finite variation and a continuous local martingale
$\wt M=(\wt M_t)_{t\in[0,\infty)}$. Then, for $t\in[0,\infty)$,
\begin{equation*}
\wt A_t=\ol A_t\quad\text{and}\quad\wt M_t=\ol M_t\quad\PP\text{-a.s. on }\{t<\zeta\},
\end{equation*}
hence \eqref{eq:sem_p2_1} and~\eqref{eq:sem_p3_2} hold.
By alternatives $(M_1)$,~$(M_2)$ and $(A_1)$,~$(A_2)$ above,
\eqref{eq:sem_dif7} and~\eqref{eq:sem_dif3} hold.
This proves Theorem~\ref{th:sem_dif3}
and the ``only~if''-part of Theorem~\ref{th:sem_dif1}.

%============================================
\newitem
\label{it:sem_p6}
In order to prove the ``if''-part of Theorem~\ref{th:sem_dif1}
we now assume that \eqref{eq:sem_dif3} holds.
By~$(A_1)$ and the reasoning in item~\ref{it:sem_p3},
\eqref{eq:sem_p3_2} and~\eqref{eq:sem_p2_1}
(which is equivalent to~\eqref{eq:sem_p3_1}) are satisfied.
Then, by items~\ref{it:sem_p2} and~\ref{it:sem_p3},
$g(Y)$ is a semimartingale with the decomposition
\begin{equation*}
g(Y_t)=g(x_0)+A_t+M_t,\quad t\in[0,\infty),
\end{equation*}
where $A$ and $M$ are given in \eqref{eq:sem_p3_1a}
and~\eqref{eq:sem_p2_2}.

Thus, Theorem~\ref{th:sem_dif1} is proved.
Theorem~\ref{th:sem_dif4} can be proved by a similar reasoning
(again use the alternatives $(M_1)$,~$(M_2)$ and $(A_1)$,~$(A_2)$
and items \ref{it:sem_p2} and~\ref{it:sem_p3}).

%============================================
\newitem
\label{it:sem_p7}
It remains to prove Theorem~\ref{th:sem_dif2}.
Let us assume that \eqref{eq:sem_dif6} is satisfied.
Then $\PP$-a.s. on $\{\zeta<\infty\}$ it holds:
\begin{equation}
\label{eq:sem_p7_1}
\text{there exists }\eps>0\text{ such that }(\ol A_t)_{t\in(\zeta-\eps,\zeta)}\text{ is monotone,}
\end{equation}
hence, $\PP$-a.s. on $\{\zeta<\infty\}$ there exist limits
\begin{equation*}
\ol A_\zeta:=\lim_{t\uparrow\zeta}\ol A_t
\quad\text{and}\quad
\ol M_\zeta:=\lim_{t\uparrow\zeta}\ol M_t,
\end{equation*}
which are either both finite or both infinite (see~\eqref{eq:sem_p1_1}).
By alternative $(M_1)$,~$(M_2)$ above,
either the limit $\lim_{t\uparrow\zeta}\ol M_t$ is finite
or \eqref{eq:sem_p4_1} holds.
Then $\ol M_\zeta$ and, consequently, $\ol A_\zeta$ are finite.
Thus, \eqref{eq:sem_p3_1} holds.

Now it follows from the fact that $\ol A$ has a locally finite variation on $[0,\zeta)$
and from \eqref{eq:sem_p7_1} and~\eqref{eq:sem_p3_1} that \eqref{eq:sem_p3_2} holds.
By alternative $(A_1)$,~$(A_2)$, we get that \eqref{eq:sem_dif3} holds,
hence, by Theorem~\ref{th:sem_dif1}, $g(Y)$ is a semimartingale.
This completes the proof.

%============================================
\appendix
\section{Bessel Process of Dimension~\protect$\delta\in(0,1)$ Is Not a Semimartingale}
\label{sec:BesNSem}
It is known that a Bessel process of dimension
$\delta\in(0,1)$
is not a semimartingale.
However, we did not find a direct reference for this.
We think this can be deduced from the general
Theorem~7.9 in~\cite{CinlarJacodProtterSharpe:80},
but this does not look straightforward.
Therefore, we now present a direct proof.

Let $x_{0}\ge0$.
Consider a squared Bessel process $Y$
of dimension $\delta\in(0,1)$
starting from $x_{0}^{2}$, i.e. $Y$ satisfies
\begin{equation}
\label{eq:BesNSem1}
Y_{t}=x_{0}^{2}+\delta t+\int_{0}^{t}2\sqrt{Y_{s}}\,\dd W_{s},
\quad t\ge0,
\end{equation}
where $W$ is a Brownian motion.
It is well-known that SDE~\eqref{eq:BesNSem1}
has a pathwise unique strong solution,
which is nonnegative, and it holds
\begin{equation}
\label{eq:BesNSem2}
\int_{0}^{\infty}I(Y_{s}=0)\,\dd s=0\quad\text{a.s.}
\end{equation}
(see~\cite[Ch.~XI, \S~1]{RevuzYor:99}).
A Bessel process of dimension $\delta\in(0,1)$
starting from~$x_{0}$ is by definition
\begin{equation*}
\rho_{t}:=\sqrt{Y_{t}},\quad t\ge0.
\end{equation*}
Assume $\rho=x_{0}+M+A$
for a continuous local martingale $M$
and a continuous finite variation process $A$
with $M_{0}=A_{0}=0$.
In particular, $\rho$ has a continuous in $t$
and c\`adl\`ag in $a$ version
$(L_{t}^{a}(\rho);t\ge0,a\in\bbR)$
of local time.
The process
$\int_{0}^{.}I(\rho_{s}=0)\,\dd M_{s}$
is a continuous local martingale
starting from $0$ with the quadratic variation
\begin{equation*}
\int_{0}^{t}I(\rho_{s}=0)\,\dd\la M,M\ra_{s}
=\int_{0}^{t}I(\rho_{s}=0)\,\dd\la \rho,\rho\ra_{s}
=\int_{\bbR}I_{\{0\}}(a)L_{t}^{a}(\rho)\,\dd a=0
\quad\text{a.s.,}\quad t\ge0,
\end{equation*}
where the second equality follows from
the occupation times formula
(see~\cite[Ch.~VI, Cor.~1.6]{RevuzYor:99}), i.e.
\begin{equation}
\label{eq:BesNSem3}
\int_{0}^{t}I(\rho_{s}=0)\,\dd M_{s}=0
\quad\text{a.s.,}\quad t\ge0.
\end{equation}
Since $Y=\rho^{2}$, we have
\begin{equation}
\label{eq:BesNSem4}
Y_{t}=x_{0}^{2}+\int_{0}^{t}2\rho_{s}\,\dd M_{s}
+\int_{0}^{t}(2\rho_{s}\,\dd A_{s}+\dd\la\rho,\rho\ra_{s}),
\quad t\ge0.
\end{equation}
Comparing decompositions~\eqref{eq:BesNSem1}
and~\eqref{eq:BesNSem4}
and using~\eqref{eq:BesNSem3} and~\eqref{eq:BesNSem2}
we obtain
\begin{equation*}
M_{t}=\int_{0}^{t}I(\rho_{s}\ne0)\,\dd M_{s}
=\int_{0}^{t}I(\rho_{s}\ne0)\,\dd W_{s}=W_{t}
\quad\text{a.s.,}\quad t\ge0.
\end{equation*}
Then $\la\rho,\rho\ra_{t}=\la M,M\ra_{t}=t$,
hence, by~\eqref{eq:BesNSem1} and~\eqref{eq:BesNSem4},
\begin{equation*}
\int_{0}^{t}2\rho_{s}\,\dd A_{s}
=(\delta-1)t,\quad t\ge0,
\end{equation*}
which yields
\begin{equation}
\label{eq:BesNSem5}
A_{t}=\int_{0}^{t}I(\rho_{s}=0)\,\dd A_{s}
+\int_{0}^{t}I(\rho_{s}\ne0)
\frac{\delta-1}{2\rho_{s}}\,\dd s
\quad\text{a.s.,}\quad t\ge0.
\end{equation}
By the occupation times formula,
for the term
$\int_{0}^{t}I(\rho_{s}\ne0)
\frac{\delta-1}{2\rho_{s}}\,\dd s$
to be finite, we necessarily have
$L_{t}^{0}(\rho)=0$~a.s., $t\ge0$.
Furthermore,
$L_{t}^{0-}(\rho)=0$~a.s., $t\ge0$,
because $\rho$ is nonnegative.
By~\cite[Ch.~VI, Th.~1.7]{RevuzYor:99},
\begin{equation*}
\int_{0}^{t}I(\rho_{s}=0)\,\dd A_{s}
=\frac12(L_{t}^{0}(\rho)-L_{t}^{0-}(\rho))=0
\quad\text{a.s.,}\quad t\ge0.
\end{equation*}
Thus, using~\eqref{eq:BesNSem5}, we get that $\rho$
is a nonnegative global (i.e. on $[0,\infty)$)
solution of the SDE
\begin{equation}
\label{eq:BesNSem6}
\dd\rho_{t}
=I(\rho_{t}\ne0)\frac{\delta-1}{2\rho_{t}}\,\dd t
+\dd W_{t}.
\end{equation}
But, by \cite[Th.~2.13]{ChernyEngelbert:05},
the latter SDE does not have a nonnegative global solution.
Here is a description of what happens:
the singular point~$0$ of SDE~\eqref{eq:BesNSem6}
has \emph{right type~1},
which is one of \emph{non-entrance types},
in the terminology of~\cite{ChernyEngelbert:05},
that is, after $\rho$ reaches~$0$,
which happens at a finite time with probability~$1$,
it cannot be continued in the positive direction
(also see \cite[Sec.~2.4]{ChernyEngelbert:05}).
The obtained contradiction completes the proof.

%============================================
\section{Behaviour of One-Dimensional Diffusions}
\label{app:Friff_Prop}
Now we state some well-known results
about the behaviour of a one-dimensional diffusion $Y$
of~\eqref{eq:set1}
with the coefficients satisfying~\eqref{eq:set2}
and~\eqref{eq:set3}.
These results follow from the construction
of solutions of~\eqref{eq:set1}
(see e.g.~\cite{EngelbertSchmidt:91}
or \cite[Ch.~5.5]{KaratzasShreve:91}
or \cite[Ch.~2 and Ch.~4]{ChernyEngelbert:05})
or can be deduced from the results
in~\cite[Sec.~1.5]{EngelbertSchmidt:89}.

\begin{proposition}
\label{prop:set1}
For any $a\in J$, with
\begin{equation*}
\tau^{Y}_{a}:=\inf\{t\ge0:Y_{t}=a\}
\qquad(\inf\emptyset:=\infty),
\end{equation*}
we have $\PP(\tau^Y_a<\infty)>0$.
\end{proposition}

Let us consider the sets
\begin{align*}
A&=\left\{\zeta=\infty,\;\limsup_{t\to\infty}Y_t=r,\;\liminf_{t\to\infty}Y_t=l\right\},\\
B_r&=\left\{\zeta=\infty,\;\lim_{t\to\infty}Y_t=r\right\},\\
C_r&=\left\{\zeta<\infty,\;\lim_{t\uparrow\zeta}Y_t=r\right\},\\
B_l&=\left\{\zeta=\infty,\;\lim_{t\to\infty}Y_t=l\right\},\\
C_l&=\left\{\zeta<\infty,\;\lim_{t\uparrow\zeta}Y_t=l\right\}.
\end{align*}

\begin{proposition}
\label{prop:set2}
Either $\PP(A)=1$ or $\PP(B_r\cup B_l\cup C_r\cup C_l)=1$.
\end{proposition}

\begin{proposition}
\label{prop:set3}
(i) $\PP(B_r\cup C_r)=0$ holds if and only if $s(r)=\infty$.

(ii) $\PP(B_l\cup C_l)=0$ holds if and only if $s(l)=-\infty$.
\end{proposition}

In particular, we get that $\PP(A)=1$ holds if and only if $s(r)=\infty$, $s(l)=-\infty$.

\begin{proposition}
\label{prop:set4}
Assume that $s(r)<\infty$.
Then either $\PP(B_r)>0$, $\PP(C_r)=0$ or $\PP(B_r)=0$, $\PP(C_r)>0$.
Furthermore, we have
\begin{equation*}
\PP\left(\lim_{t\uparrow\zeta}Y_t=r,\;\;Y_t>a\;\forall t\in[0,\zeta)\right)>0
\end{equation*}
for any $a<x_0$.
\end{proposition}

\begin{proposition}[Feller's test for explosions]
\label{prop:set5}
We have $\PP(B_r)=0$, $\PP(C_r)>0$ if and only if
\begin{equation*}
s(r)<\infty\quad\text{and}\quad\frac{s(r)-s}{\rho\sigma^2}\in L^1_\loc(r-).
\end{equation*}
\end{proposition}

Clearly, Propositions \ref{prop:set4} and~\ref{prop:set5},
which contain statements about the behaviour of one-dimensional diffusions
at the endpoint~$r$, have their analogues for the behaviour at~$l$.
Feller's test for explosions in this form is taken 
from \cite[Sec.~4.1]{ChernyEngelbert:05}.
For a different (but equivalent) form see e.g. \cite[Ch.~5, Th.~5.29]{KaratzasShreve:91}.

Let us finally emphasize
that the results stated in this appendix
do not in general hold
beyond~\eqref{eq:set2} and~\eqref{eq:set3}.

%============================================
%\nocite{*}
%\bibliographystyle{chicago}
%\bibliographystyle{amsplain}
%\bibliographystyle{plain}
\bibliographystyle{abbrv}
\bibliography{refs}
\end{document}